\documentclass[twoside]{article}
\usepackage{subcaption}
\usepackage[accepted]{tmp2025}
\usepackage[round]{natbib}

\usepackage{amsmath, amsfonts, amssymb, amsthm}
\usepackage{enumitem}
\usepackage{makecell}
\usepackage{color}
\usepackage{booktabs}
\usepackage{mathtools}
\mathtoolsset{showonlyrefs}
\usepackage{hyperref}
\hypersetup{
    colorlinks=true,
    linkcolor=blue,
    citecolor=magenta,      
    urlcolor=cyan,
    pdfpagemode=FullScreen,
}

\newcommand{\colsp}{\mathrm{colsp}}
\newcommand{\rowsp}{\mathrm{rowsp}}

\newcommand{\cV}{\mathcal{V}}

\newcommand{\NN}{\mathbb{N}}
\newcommand{\RR}{\mathbb{R}}
\newcommand{\Rb}{\mathbb{R}}
\newcommand{\bone}{\boldsymbol{\mathbf{1}}}
\newcommand{\bbE}{\mathbb{E}}

\newcommand{\bc}{\boldsymbol{c}}
\newcommand{\bd}{\boldsymbol{d}}
\newcommand{\be}{\boldsymbol{e}}

\newcommand{\bh}{\boldsymbol{h}}

\newcommand{\bk}{\boldsymbol{k}}

\newcommand{\bt}{\boldsymbol{t}}

\newcommand{\bv}{\boldsymbol{v}}
\newcommand{\bw}{\boldsymbol{w}} 
\newcommand{\bx}{\boldsymbol{x}}
\newcommand{\by}{\boldsymbol{y}} 
 
\newcommand{\bA}{\boldsymbol{A}}
\newcommand{\bB}{\boldsymbol{B}}

\newcommand{\bD}{\boldsymbol{D}}

\newcommand{\bG}{\boldsymbol{G}}
\newcommand{\bH}{\boldsymbol{H}}
\newcommand{\bI}{\boldsymbol{I}}

\newcommand{\bM}{\boldsymbol{M}}

\newcommand{\bO}{\boldsymbol{O}}
\newcommand{\bP}{\boldsymbol{P}}
\newcommand{\bQ}{\boldsymbol{Q}}

\newcommand{\bS}{\boldsymbol{S}}
\newcommand{\bT}{\boldsymbol{T}}
\newcommand{\bU}{\boldsymbol{U}}
\newcommand{\bV}{\boldsymbol{V}}
\newcommand{\bW}{\boldsymbol{W}} 
\newcommand{\bX}{\boldsymbol{X}}

\newcommand{\bZ}{\boldsymbol{Z}}

\newcommand{\bSigma}{\boldsymbol{\Sigma}} 

\newcommand{\balpha}{\boldsymbol{\alpha}} 
\newcommand{\bbeta}{\boldsymbol{\beta}} 
\newcommand{\bvarepsilon}{\boldsymbol{\varepsilon}} 
\newcommand{\blambda}{\boldsymbol{\lambda}}
 
\newcommand{\btau}{\boldsymbol{\tau}} 
\newcommand{\bdelta}{\boldsymbol{\delta}}
\newcommand{\bphi}{\boldsymbol{\phi}}
\newcommand{\bpsi}{\boldsymbol{\psi}}
 
\newcommand{\bzero}{\boldsymbol{0}} 
\newcommand{\btheta}{\boldsymbol{\theta}} 
\newcommand{\bgamma}{\boldsymbol{\gamma}}

\newcommand{\bDelta}{\boldsymbol{\Delta}} 
 
\newcommand{\Gram}{\bG}

\newcommand{\cI}{\mathcal{I}}
\newcommand{\cJ}{\mathcal{J}}

\newcommand{\cS}{\mathcal{S}}

\newcommand{\cW}{\mathcal{W}}

\newcommand{\htau}{\widehat{\tau}} 
\newcommand{\hbtau}{\widehat{\btau}} 
\newcommand{\hbdelta}{\widehat{\bdelta}} 
 
\newcommand{\hbbeta}{\widehat{\bbeta}} 
\newcommand{\hbgamma}{\widehat{\bgamma}}

\newcommand{\hby}{\widehat{\by}}

\newcommand{\hbtheta}{\widehat{\btheta}} 
\newcommand{\hsigma}{\widehat{\sigma}} 

\newcommand{\hbalpha}{\widehat{\balpha}} 
\newcommand{\hblambda}{\widehat{\blambda}} 
\newcommand{\hbphi}{\widehat{\bphi}}
\newcommand{\hbpsi}{\widehat{\bpsi}}
\newcommand{\hbDelta}{\widehat{\bDelta}} 

\newcommand{\tbvarepsilon}{\widetilde{\bvarepsilon}} 
\newcommand{\tvarepsilon}{\widetilde{\varepsilon}} 
\newcommand{\tbW}{\widetilde{\bW}} 
 
\newcommand{\tbalpha}{\widetilde{\balpha}}
\newcommand{\tbtau}{\widetilde{\btau}}
\newcommand{\ttau}{\widetilde{\tau}}
\newcommand{\tbtheta}{\widetilde{\btheta}}

\newcommand{\diag}{\mathrm{diag}} 
\newcommand{\rank}{\mathrm{rank}\,}

\newcommand{\tr}{\text{\normalfont tr}\,}

\newcommand{\Cov}{\text{\normalfont Cov}}

\newcommand{\si}{\sim i}

\newcommand{\F}{\textup{F}}

\newcommand{\Normal}{\mathsf{N}}

\DeclareMathOperator*{\argmin}{\arg \min}

\newtheorem{theorem}{Theorem}
\newtheorem{proposition}{Proposition}
\newtheorem{lemma}{Lemma}
\newtheorem{definition}{Definition}

\newtheorem{remark}{Remark}
\newtheorem{corollary}{Corollary}

\newcommand{\hsigmareg}{\hsigma_{\cW}}
\newcommand{\hsigmaunreg}{\hsigma_{\cW^c}}
\newcommand{\hsigmaf}{\hsigma_{\text{full}}}
\newcommand{\hsigmafin}{\hsigma_{\emph{full}}}
\newcommand{\hsigmap}{\hsigma_{\text{partial}}}
\newcommand{\hsigmapin}{\hsigma_{\emph{partial}}}

\begin{document}
\runningtitle{}
\runningauthor{}

\twocolumn[

\tmptitle{Algebraic and Statistical Properties of\\
the Partially Regularized Ordinary Least Squares Interpolator}

\tmpauthor{Letian Yang \And Dennis Shen}

\tmpaddress{University of Southern California \And  University of Southern California}]

\begin{abstract}
Modern deep learning has revealed a surprising statistical phenomenon known as benign overfitting, with high-dimensional linear regression being a prominent example. 
This paper contributes to ongoing research on the ordinary least squares (OLS) interpolator, focusing on the partial regression setting, where only a subset of coefficients is implicitly regularized. 
On the algebraic front, we extend Cochran's formula and the leave-one-out residual formula for the partial regularization framework. 
On the stochastic front, we leverage our algebraic results to design several homoskedastic variance estimators under the Gauss-Markov model. These estimators serve as a basis for conducting statistical inference, albeit with slight conservatism in their performance. 
Through simulations, we study the finite-sample properties of these variance estimators across various generative models.

\end{abstract}

\section{Introduction}
The study of overparameterized statistical models in deep learning methodology has revealed a fascinating phenomenon known as \textit{benign overfitting} \citep{zhang2017understanding, belkin_2021, benign_overfitting_tutorial}, where models that perfectly fit to (noisy) training data are still able to accurately predict on unseen test data. 
This discovery has challenged conventional wisdoms in statistical learning theory on the tradeoffs between the complexity of the prediction model and its fit to training data. 
Consequently, there has been significant theoretical interest in both statistics and machine learning to rethink the concept of generalization and the conditions under which it is achievable.  
One prominent example of a model that exhibits benign overfitting is the ordinary least squares (OLS) estimator when its number of parameters $p$ exceeds the number of training examples $n$. 
Given its simplicity and broad applicability across a myriad of domains, the OLS estimator emerges as a critical instrument to gain foundational insights into when a model can overfit to training data without harming its prediction accuracy.

\subsection{The OLS Estimator}
We assume access to training (or in-sample) data $\{(\bx_i, y_i): i =1, \dots, n\}$, where $\bx_i \in \Rb^p$ and $y_i \in \Rb$ represent the $i$-th covariate and response, respectively. Let $\bX \in \mathbb{R}^{n \times p}$ collect the covariates and $\by \in \mathbb{R}^n$ collect the responses. 
We define 
$\cS = \argmin_{\bbeta \in \Rb^p} \| \by - \bX \bbeta \|_2^2$ 
as the set of all vectors that minimize the $\ell_2$-size of the training error. 
In view of $\cS$, we can define a class of minimum $\ell_2$-norm solutions: given a nonempty index set $\cW \subseteq [p] = \{1, \dots, p\}$, let 
\begin{align}
	\hbbeta^{[\cW]} \in \argmin_{\bbeta \in \cS} \sum_{i \in \cW} \beta_i^2.  \label{eq:min_l2_norm_sol}
\end{align} 
Among all solutions in $\cS$, $\hbbeta^{[\cW]}$ yields the solution that minimizes the $\ell_2$-size of the coefficients indexed by $\cW$. 
Thus, when $\cW = [p]$, we refer to the solution as the ``fully regularized'' OLS estimator. 
By contrast, when $\cW \subset [p]$ is a proper subset of $[p]$, we refer to the solution as the ``$\cW$-partially regularized'' OLS estimator.

\paragraph{Classical regime.} 
In the classical regime, defined by $n > p$, it is commonplace to assume that $\bX$ has full column rank, i.e., $\rank(\bX) = p$.  
Accordingly, the set $\cS$ is a singleton so the fully regularized and partially regularized OLS estimators exactly coincide. 
The resulting solution is often written as 
\begin{align}
	\hbbeta &= (\bX^\top \bX)^{-1} \bX^\top \by. \label{eq:beta.classical} 
\end{align}  
Algebraic and statistical properties of \eqref{eq:beta.classical} have been extensively studied in the literature. 

\paragraph{High-dimensional regime.} 
This article focuses on the high-dimensional regime, defined by $p > n$, in which $\bX$ is typically assumed to have full row rank, i.e., $\rank(\bX) = n$. 
In this setting, $\cS$ admits infinitely many solutions with each solution perfectly fitting or \textit{interpolating} the in-sample data. 
When $\cW = [p]$, the fully regularized OLS solution is generally written as 
\begin{align}
    \hbbeta \coloneqq \hbbeta^{[[p]]} = \bX^{\top} (\bX \bX^{\top})^{-1} \by. \label{eq:high.full}
\end{align}
There has been significant progress made in recent years towards understanding the solution given by \eqref{eq:high.full}, which we will henceforth refer as the fully regularized OLS interpolator (or simply, the OLS interpolator for short), under a particular set of stochastic assumptions on the data generating process (DGP) \citep{bartlett2020benign, belkin20, muthukumar2020harmless, hastie_22, mallinar2024minimumnorm, shen2024}. 
The primary focus of most of these works is the out-of-sample risk of the OLS interpolator.

Compared to its fully regularized counterpart, the partially regularized OLS interpolator has received far less attention in the literature. 
To study this estimator, let us first partition $\bX$ into $\bX = [\bW, \bT]$, where $\bW \in \Rb^{n \times q}$ and $\bT \in \Rb^{n \times m}$ with $q + m = p$.
We now place the following assumption on $\bX$: 
\begin{enumerate}[label=(\texttt{A\arabic*}), leftmargin=2em]
    \setcounter{enumi}{0}
    \item \label{assump:partial} 
    Let $\bW$ and $\bT$ have full row rank and full column rank, respectively. 
    That is, $\rank(\bW) = n < q$ and $\rank(\bT) = m < n$. 
\end{enumerate}
Note that Assumption~\ref{assump:partial} is a stronger assumption than the canonical full row rank assumption of $\bX$. 
As shown in \citet{shen2024}, when Assumption~\ref{assump:partial} holds, then the $\cW$-partially regularized OLS interpolator has the following decompositional form: 
\begin{align}
	\hbbeta^{[\cW]} &= 
	\begin{pmatrix}
		\hbbeta^{[\cW]}_{\cW}
		\\
		\hbbeta^{[\cW]}_{\cW^c}
	\end{pmatrix} 
	= \begin{pmatrix}
		(\bP_{\bT}^{\perp} \bW)^{\dagger} \bP_{\bT}^{\perp}
		\\
		(\bW^\dagger \bT)^\dagger \bW^\dagger
	\end{pmatrix} 
	\by, \label{eq:beta.fwl.hd} 
\end{align}
where $\hbbeta^{[\cW]}_{\cW}$ and $\hbbeta^{[\cW]}_{\cW^c}$ are the coefficients indexed by $\cW$ and its complement set $\cW^c$, respectively; 
$\dagger$ is the Moore-Penrose pseudoinverse;
$\bP_{\bT} \coloneqq \bT \bT^{\dagger}$ is the projection matrix onto the column space of $\bT$ and $\bP_{\bT}^\perp$ its orthogonal complement. 
In words, \eqref{eq:beta.fwl.hd} is obtained by penalizing the coefficients of $\bW$ and excluding the coefficients of $\bT$ from regularization. 
The decompositional form given in \eqref{eq:beta.fwl.hd} can be viewed as an extension of the Frisch-Waugh-Lovell (FWL) theorem \citep{frisch_waugh, lovell} from the classical regime to the high-dimensional regime. 

\begin{remark} [Alternative expression] \label{rmk:fwl.j.hd.1}
  \citet{shen2024} showed an alternative expression for the regularized coefficients in \eqref{eq:beta.fwl.hd}:  $\hbbeta^{[\cW]}_{\cW}=\bP_{\bW^{\top}} \bP_{\bW^{\dagger} \bT}^{\perp} \bW^{\dagger} \by$. 
  This expression is frequently used throughout the proofs in this paper.    
\end{remark} 

\subsection{Why Partial Regularization?}\label{sec:intro.moti}

The primary focus of this article is to study the partially regularized OLS interpolator given by \eqref{eq:beta.fwl.hd}. 
We motivate our study with two illustrative examples. 

\paragraph{Ridge regression.}
Consider ridge regression with $\by$ and $\bX = [\bW,\bone] \in \Rb^{n \times p}$, where $\bone$ is the $n$-vector of all ones and $\bW$ represents the covariates. 
Formally, for any $\lambda > 0$, the ridge solution is defined as  
\begin{align}
	\hbbeta^\lambda &= \argmin_{\alpha_0 \in \Rb, \balpha \in \Rb^{q}} \| \by - ( \bW \balpha+\alpha_0 \bone) \|_2^2 + \lambda \| \balpha \|_2^2. \label{eq:ridge} 
\end{align}
As noted in \citet{elements_of_stat_learning}, the intercept, $\alpha_0$, is not penalized in the ridge formulation of \eqref{eq:ridge} to prevent a dependency on the choice of origin for $\by$. 
Since the minimum $\ell_2$-norm OLS estimator is equivalent to the limit of ridge regression as the tuning parameter $\lambda$ tends to zero, it is reasonable to also exclude the intercept in the least squares formulation of \eqref{eq:min_l2_norm_sol}, in which $\cW$ indexes the coefficients of $\bW$; this leads to the solution given by \eqref{eq:beta.fwl.hd}, where $\cW^c$ indexes the coefficients of $\bT = \bone$. 

\paragraph{Treatment effect estimation.} 
As further motivation, we turn to a well-studied problem from the causal inference literature. 
Consider $n$ units. 
For each unit $i \in [n]$, let $\bw_i \in \Rb^q$ denote the vector of covariates, $Y_i(1), Y_i(0) \in \Rb$ denote the potential outcomes with and without treatment, respectively, and $D_i \in \{0,1\}$ denote the binary treatment indicator. 
Let $Y_i = D_i Y_i(1) + (1-D_i) Y_i(0)$ denote the observed outcome for unit $i$. 
The objective is to estimate the average treatment effect (ATE) defined by 
$\tau = \bbE[Y(1) - Y(0)]$. 

A canonical approach is to posit a linear outcome model \citep{ding2023coursecausalinference}: 
$\bbE[\by | \bD, \bW] =  \bW \balpha+\tau \bD + \alpha_0 \bone$, 
where $\by$, $\bW$, and $\bD$ collect the observed outcomes, covariates, and treatment assignments, respectively. 
Practitioners often regress $\by$ on $[\bW, \bD, \bone]$ and define the ATE estimate as the coefficient associated with $\bD$. 
This problem is well studied in the classical regime where $p = q+2 < n$, but has yet to be explored in the high-dimensional regime where $p > n$. 
As aforementioned, there are two distinct regression-based estimates in the high-dimensional setting: full regularization and partial regularization. 
In the former, all regression coefficients are penalized while in the latter, only the regression coefficients associated with the covariates, $\bW$, are penalized, i.e., under partial regularization, $\bT = [\bD, \bone] \in \{0,1\}^{n \times 2}$. 

We conduct a simulation based on the linear outcome model and present the biases induced by the fully regularized and partially regularized estimates in Figure~\ref{fig:motivating_eg}. 
Our empirical results demonstrate that the partially regularized OLS interpolator generally yields significantly lower bias compared to its fully regularized counterpart. 
To the best of our knowledge, the partially regularized OLS interpolator is a novel approach for the high-dimensional analog of this age-old problem and warrants further investigation. 
This example also highlights a common objective in causal inference applications: parameter estimation, which stands in contrast to many machine learning tasks that view prediction as the primary object of study. 
Similar to the discussion above, this example highlights the importance of appropriately defining $\bW$ and $\bT$ as their constructions can have substantive implications. 
\begin{figure}[t!]
    \vspace{.15in}
    \centering
    \includegraphics[width=.75\linewidth]{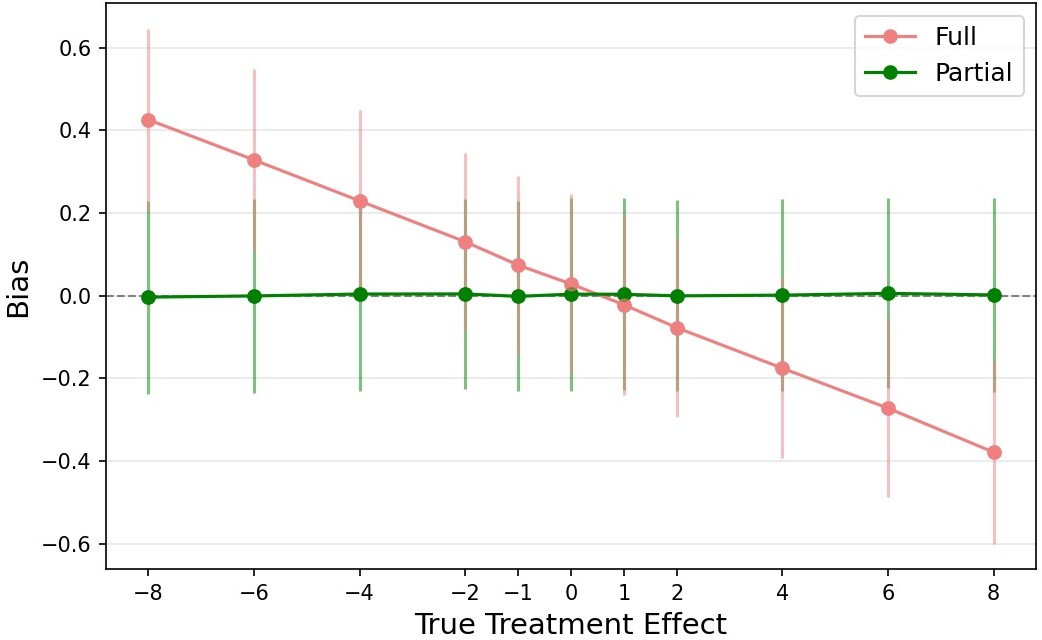}
    \vspace{.3in}
    \caption{Biases in estimating the ATE for fully (red) and partially (green) regularized OLS interpolators. Here, $(n,p)=(80,100)$, $\tau = [0, \pm 1, \pm 2, \pm 4, \pm 6, \pm 8]$, $\bW$ is generated from a spiked covariance model, $\bD \sim \text{Bernoulli}(0.5)^n$, and standard normal noise is added. See supplementary materials for details.}\label{fig:motivating_eg}
    \vspace{-.1in}
\end{figure}

\subsection{Contributions and Organization}
Our results can be divided into two parts: algebraic and statistical. 
The first set of results are derived purely from algebraic principles, independent of any stochastic assumptions on the DGP. 
To arrive at our second set of results, we anchor on a particular DGP and leverage our algebraic results from part one. 

\paragraph{Algebraic results.} In Section \ref{sec:col_par}, we provide a high-dimensional, partially regularized counterpart to Cochran's formula (Theorem \ref{thm:cochran's}). We show that Cochran's formula, originally established in the classical regime, continues to hold in the high-dimensional regime under partial regularization. We discuss implications of this result for treatment effect estimation in observational studies when critical variables are omitted from the OLS model. 
In Section \ref{sec:row_par}, we present closed-form leave-one-out (LOO) formulas (Proposition \ref{prop:loo}) and residuals (Corollary \ref{cor:loo.residual}) in the high-dimensional partially regularized setting. These results provide insights into the generalization behavior of the partially regularized OLS interpolator and help clarify the similarities and differences between the fully and partially regularized OLS interpolators.

\paragraph{Stochastic results.} In Section \ref{sec:stat_infer}, we study the partially regularized OLS interpolator under the Gauss-Markov model. 
In particular, we design three homoskedastic variance estimators and analyze their biases (Theorems~\ref{thm:var.estimator.pr} and \ref{thm:var.estimator.partial}). 
In Section \ref{sec:simulation}, we investigate the finite-sample performance of these estimators through illustrative simulation studies under various generative models and from different perspectives, offering practical insights for practitioners. 

In Section \ref{sec:conclusion}, we summarize our findings and discuss several avenues of future research. 


\paragraph{Notation.} For any matrix $\bM \in \Rb^{n \times p}$, let $\bP_{\bM} \coloneqq \bM \bM^{\dagger}$ be the projection matrix onto its column space and $\bP_{\bM}^{\perp} = \bI_n - \bP_{\bM}$ the projection onto its orthogonal complement. Denote $\Gram_{\bM} = (\bM \bM^\top)^{\dagger} \in \Rb^{n \times n}$ as the ``inverse Gram matrix" of $\bM$. 
Let $\be_i$ denote the $i$-th standard basis vector in $\Rb^n$ and $\diag(\bM) \coloneqq \sum_{i=1}^n M_{ii} \be_i \be_i^\top$ for $\bM \in \Rb^{n \times n}$.

\section{Cochran's formula}\label{sec:col_par}
This section establishes the high-dimensional analog of Cochran's formula under the partial regularization framework. 
Towards this, we recall Assumption~\ref{assump:partial}. 
We then decompose $\bW$ as $\bW = [\bZ, \bU]$, where $\bZ \in \Rb^{n \times \ell}$ and $\bU \in \Rb^{n \times r}$ with $\ell + r = q$. 
We state an added assumption that implies Assumption~\ref{assump:partial}: 
\begin{enumerate}[label=(\texttt{A\arabic*}), leftmargin=2em]
    \setcounter{enumi}{1}
    \item \label{assump:partial_3} 
    Let $\bT$ retain its full row rank. 
    Further, let $\bZ$ have full row rank and $\bU$ have full column rank, i.e., $\rank(\bZ) = n < \ell$ and $\rank(\bU) = r < n$. 
\end{enumerate}
Consider the following OLS formulations: 
\begin{align}
    \cS_1 &\coloneqq \argmin_{\balpha \in \Rb^\ell, \bgamma \in \Rb^r, \btau \in \Rb^m} \| \by - (\bZ \balpha + \bU \bgamma + \bT \btau) \|_2^2, \label{eq:column.s1} \\
    \cS_2 &\coloneqq \argmin_{\balpha \in \Rb^\ell, \btau \in \Rb^m} \| \by - (\bZ \balpha + \bT \btau)\|_2^2, \label{eq:column.s2} \\
    \cS_3 &\coloneqq \argmin_{\bDelta \in \Rb^{\ell \times r}, \bdelta \in \Rb^{m \times r}} \| \bU - (\bZ \bDelta + \bT \bdelta) \|_F^2, \label{eq:column.s3}
\end{align} 
where $\| \cdot \|_F$ denotes the Frobenius norm. 
In the high-dimensional regime, $\cS_1$ to $\cS_3$ contain infinitely many solutions. 
As such, we define the associated partially regularized OLS solutions as 
\begin{align}
\begin{pmatrix}
    \hbalpha \\
    \hbgamma \\
    \hbtau
\end{pmatrix}
= \argmin_{(\balpha, \bgamma, \btau) \in \cS_1} \left( \|\balpha\|_2^2 + \|\bgamma\|_2^2 \right), 
\\
\begin{pmatrix}
    \tbalpha \\
    \tbtau
\end{pmatrix}
= \argmin_{(\balpha, \btau) \in \cS_2}  \|\balpha\|_2^2 , ~ 
\begin{pmatrix}
    \hbDelta \\
    \hbdelta
\end{pmatrix}
= \argmin_{(\bDelta, \bdelta) \in \cS_3} & \|\bDelta\|_F^2, \label{eq:partial_in_set}
\end{align}
which all exclude the coefficients of $\bT$ from penalization. 
Armed with these definitions, we are ready to state the following result. 

\begin{theorem}[Cochran's formula] \label{thm:cochran's} 
If Assumption \ref{assump:partial_3} holds, then for any tuples $(\hbalpha, \hbgamma, \hbtau) \in \cS_1$, $(\tbalpha, \tbtau) \in \cS_2$, and $(\hbDelta, \hbdelta) \in \cS_3$, we have
\begin{equation}
    \bZ \tbalpha + \bT \tbtau = \bZ (\hbalpha + \hbDelta \hbgamma) + \bT (\hbtau + \hbdelta \hbgamma). \label{eq:cochran_image}
\end{equation}
If the tuples $(\hbalpha, \hbgamma, \hbtau)$, $(\tbalpha, \tbtau)$, and $(\hbDelta, \hbdelta)$ are each the partially regularized OLS estimators defined in \eqref{eq:partial_in_set}, then we further have 
\begin{align}
    \begin{pmatrix}
        \tbalpha \\
        \tbtau
    \end{pmatrix}
    =
     \begin{pmatrix}
        \hbalpha \\
        \hbtau
    \end{pmatrix}
	+
    \begin{pmatrix}
        \hbDelta \\
        \hbdelta
    \end{pmatrix}
    \hbgamma. \label{eq:cochran}
\end{align}    
\end{theorem}
Theorem~\ref{thm:cochran's} is a purely linear algebraic fact that is independent of any statistical model. 
Moreover, \eqref{eq:cochran_image} and \eqref{eq:cochran} can be viewed as extensions of Cochran's formula, attributed to \citet{cochran} and established in the classical regime, to the high-dimensional setting. 
In this view, \eqref{eq:cochran_image}, which states that the projections onto the images of $\bZ$ and $\bT$ are preserved for any of the infinitely many solutions in the sets $\cS_1$ to $\cS_3$, can be interpreted as the predictive counterpart of Cochran's formula. 
Cochran's formula, originally proven for the regression coefficients, can be reinstated in the high-dimensional setting provided the solutions are partially regularized as per the constraints in \eqref{eq:partial_in_set}. 

\paragraph{Example application: treatment effect estimation.}
Within the classical regime, econometricians famously refer to Cochran's formula as the {\em omitted-variable} (OVB) bias formula.
%
To describe the OVB problem, recall the example from Section~\ref{sec:intro.moti} that aims to estimate the ATE. 
A challenge in estimating the ATE in observational studies is unmeasured confounding, i.e., when a crucial set of variables, $\bU$, with respect to which the treatment and control units may differ, are unobserved and thus omitted from the OLS model.
Consequently, the OLS coefficient of the treatment assignments, $\bD$, is biased when the unobserved confounders $\bU$ are related to both $\bD$ and outcomes $\by$, even after conditioning on the covariates. 
Our simulation results in Figure~\ref{fig:motivating_eg} suggest that partial regularization, compared to full regularization, is the more suitable  regression formulation in the high-dimensional setting. 
Thus, Theorem~\ref{thm:cochran's} provides a quantification of the bias induced by omitting $\bU$. 
\begin{corollary} \label{cor:ovb_cochran}
Let $\bZ$ denote the observed confounders, $\bU$ the unobserved confounders, $\bT = [\bD, \bone]$ the collection of treatment assignments and intercept, and $\by$ the observed outcomes. 
Then, \eqref{eq:cochran_image} and \eqref{eq:cochran} of Theorem \ref{thm:cochran's} hold under appropriate conditions.
\end{corollary}
With $\bT = [\bD, \bone]$, let $\htau_D$, $\ttau_D$, and $\hbdelta_D$ denote the coefficients of $\bD$ for $\cS_1$, $\cS_2$, and $\cS_3$, respectively, as defined in \eqref{eq:partial_in_set}. 
Corollary \ref{cor:ovb_cochran} then states that $\ttau_D - \htau_D = \hbdelta_D \hbgamma$.
This indicates that the bias in the treatment effect estimates is influenced by two factors: $\hbgamma$ describes how different subgroups of the unobserved confounder \textit{impact} the outcome prediction, while $\hbdelta_d$ captures the \textit{imbalance} between the treatment and control groups in terms of the linear expectation of the unobserved confounder.

\section{Leave-One-Out Analysis} \label{sec:row_par}
In this section, we focus on the leave-one-out (LOO) configuration, which is a useful framework to study the stability of an estimator. 

\paragraph{Notation.} 
We inherit the setting of Assumption~\ref{assump:partial} to recall \eqref{eq:beta.fwl.hd}.  
To avoid cluttered notation, let $\hblambda \coloneqq \hbbeta^{[\cW]}_{\cW}$ and $\hbtau \coloneqq \hbbeta^{[\cW]}_{\cW^c}$. 
For any $i \in [n]$, we denote $\bW_{\sim i}$, $\bT_{\sim i}$, and $\by_{\sim i}$ as the leave-$i$-out data with the $i$-th row/entry removed. 
Let $\hblambda^{(\sim i)}$ and $\hbtau^{(\sim i)}$ denote the partially regularized solutions from regressing $\by_{\sim i}$ on $[\bW_{\sim i}, \bT_{\sim i}]$.  
 
\subsection{LOO Regression Coefficients}
\begin{proposition}\label{prop:loo} 
For any $i \in [n]$, denote  
\begin{align}
    \bP_{\bW^\dagger \be_i} = \frac{\bW^\dagger \be_i \be_i^\top \bW^{\dagger,\top}}{\be_i^\top \Gram_{\bW} \be_i}  \label{mat:Pi}
\end{align}
as the projection matrix onto the column space of $\bW^\dagger \be_i$ and $\bP^\perp_{\bW^\dagger \be_i}$ its orthogonal complement. 
Let $\tbW_i = \bP_{\bW^\dagger \be_i}^\perp \bW^\dagger$.  
If Assumption~\ref{assump:partial} holds, then
\begin{align}
    \hblambda^{(\sim i)} &= \bW^\top \tbW_i^\top \bP_{\tbW_i\bT}^\perp \tbW_i \by \label{eq:loo_beta} 
    \\
    \hbtau^{(\sim i)} &= (\tbW_i\bT)^\dagger \tbW_i \by. \label{eq:loo_tau} 
\end{align}
\end{proposition} 
By replacing $\bW^\dagger$ with its projection onto the orthogonal complement of $\text{colsp}(\bW^\dagger \be_i)$, i.e., $\tbW_i$, in the full-sample estimator formulas in \eqref{eq:beta.fwl.hd}, we obtain the corresponding leave-$i$-out estimators.
For a comparison with the regularized coefficients, recall the alternative expression in Remark \ref{rmk:fwl.j.hd.1} and note that $\bP_{\bW^\top} \bP_{\bW^\dagger \bT}^\perp \bW^\dagger \by = \bW^{\top} \bW^{\top, \dagger} \bP_{\bW^\dagger \bT}^\perp \bW^\dagger \by$. 
%

\subsection{LOO Prediction Residuals}
LOO prediction residuals are useful for assessing how well a model generalizes to unseen data. We represent the leave-$i$-out prediction residual as $$\tvarepsilon_i\coloneqq y_i - (\bw_i^{\top} \hblambda^{(\sim i)} + \bt_i^\top \hbtau^{(\sim i)}),$$ where $\bw_i^\top$ and $\bt_i^\top$ denote the $i$-th rows of $\bW$ and $\bT$, respectively. By leveraging Proposition \ref{prop:loo}, we can derive the formula for the LOO prediction residuals in the partial regularization setting.
\begin{corollary} \label{cor:loo.residual} 
Let $\bH_i = \bT (\tbW_i \bT)^\dagger \tbW_i$, 
where $\tbW_i$ is defined as in Proposition \ref{prop:loo}. If Assumption~\ref{assump:partial} holds, then we have
\begin{equation}
    \tbvarepsilon_i 
    =\be_i^\top \big[ \diag(\Gram_{\bW}) \big]^{-1}  \Gram_{\bW} (\bI_n - \bH_i)  \by.  \label{eq:loo_single}
\end{equation}
\end{corollary}
To contextualize Corollary~\ref{cor:loo.residual}, we first note that the LOO residuals in the classical regime can be expressed as $\tbvarepsilon_{\text{classical}} = [\diag(\bP_{\bX}^\perp)]^{-1} \bP_{\bX}^\perp \by$, provided $\bX$ has full column rank. 
Similarly, \citet[Corollary 3]{shen2024} showed that the LOO residuals in the high-dimensional regime under full regularization obey 

\begin{align}
    \tbvarepsilon_{\text{full}} &= \big[ \diag(\Gram_{\bX}) \big]^{-1}  \Gram_{\bX}  \by,  \label{eq:loo.shortcut.hd}
\end{align}
provided $\bX$ has full row rank. 
Corollary~\ref{cor:loo.residual} demonstrates that the LOO residuals in the high-dimensional partial regularization setting admit a similar closed-form expression, particularly with respect to $\tbvarepsilon_{\text{full}}$, albeit with more complexity. 
This complexity is likely an artifact of the partially regularized OLS interpolator treating $\bW$ and $\bT$ as distinct objects, which then leads to a unique term, $\bH_i$, for the $i$-th LOO residual. 
The gap between the two expressions highlights the added complication of the partial regularization setting compared to the full regularization setting. 
Nevertheless, we hope the element-wise closed-form residual formulas of Corollary~\ref{cor:loo.residual} may serve as a useful tool to formally analyze the partially regularized OLS interpolator's generalization capabilities under particular statistical frameworks. 
Additionally, as we will see in the following section, the result in Corollary~\ref{cor:loo.residual} enables us to construct a variance estimator under homoskedasticity (see Theorem~\ref{thm:var.estimator.pr}).

\section{Statistical Inference} \label{sec:stat_infer}
Up until this point, we have not made any stochastic assumptions. 
In other words, all results presented in Sections~\ref{sec:col_par} and \ref{sec:row_par} are purely algebraic and hold for any DGP.
Now, we will postulate the canonical Gauss-Markov model under homoskedasticity \citep{aitken1936iv}. 

\begin{enumerate}[label=(\texttt{B\arabic*})]
	\item \label{assump:lm} 
	Let $\by = \bX \bbeta + \bvarepsilon$, where $\bX \in \Rb^{n \times p}$ is fixed, $\bbeta \in \Rb^p$ is a deterministic set of model parameters, and $\bvarepsilon \in \Rb^n$ is a random vector with $\bbE[\bvarepsilon] = \bzero$ and $\Cov(\bvarepsilon) = \sigma^2 \bI_n$. 	
\end{enumerate} 
The focus of this section is to recover $\sigma^2$, which represents an important quantity in many problems in statistics and machine learning. 
In prediction tasks, $\sigma^2$ often determines the scale of an estimator's risk under the $\ell_2$-loss. 
Popular model selection statistics, including AIC and BIC, rely on reliable estimates of $\sigma^2$. 
Estimating the signal-to-noise ratio, frequently measured as $\| \bbeta \|_2 / \sigma^2$, is critical for selecting the tuning parameter (e.g., in ridge and Lasso) and can determine the performance limits in certain high-dimensional regression problems. 
Finally, conducting inference in various causal inference settings often requires consistent estimation of $\sigma^2$. 
With this in mind, we proceed to propose three variance estimators for $\sigma^2$.

\subsection{Variance Estimation via LOO Residuals}\label{sec:ve_loo}
The de facto approach within the classical regime of $n > p$ is to estimate $\sigma^2$ via the OLS estimator's in-sample residuals; specifically, $\hsigma^2 = (n-p)^{-1} \| \bP_{\bX}^\perp \by \|_2^2$. 
This method enjoys many properties including unbiasedness and consistency. 
However, such an approach fails to translate to the high-dimensional regime since the OLS estimator interpolates the data, which leads to zero in-sample residuals. 

To overcome this challenge, the recent work of \citet{shen2024} proposed a variance estimator based on the LOO residuals under the \textit{full regularization} setting with the assumption that $\bX$ has full row rank:  
\begin{align}
    \hsigma_{\text{full}}^2 &= \frac{ \left\| \tbvarepsilon_{\text{full}} \right\|_2^2 }{ \big\| \big[ \diag(\Gram_{\bX}) \big]^{-1} \Gram_{\bX} \big\|_{\F}^2 }, \label{eq:var.estimator.hd} 
\end{align}
where $\tbvarepsilon_{\text{full}}$ is defined as in \eqref{eq:loo.shortcut.hd}. 
Under Assumption~\ref{assump:lm}, it follows that
\begin{align}
    \bbE[\hsigma_{\text{full}}^2] = \sigma^2 + \frac{ \big\| \bbE[\tbvarepsilon_{\text{full}}] \big\|_2^2}{ \big\| \big[ \diag(\Gram_{\bX}) \big]^{-1}  \Gram_{\bX} \big\|_{\F}^2 }. \label{eq:bias.hd} 
\end{align} 
Following this LOO strategy, we anchor on Corollary \ref{cor:loo.residual} to construct an analogous variance estimator in the high-dimensional partially regularized setting. 

\begin{theorem} \label{thm:var.estimator.pr} 
Let Assumption~\ref{assump:partial} hold. 
Define 
\begin{equation}
    \hsigmapin^2 = \frac{\| \tbvarepsilon_{\emph{partial}}^2\|_2^2 }{\Sigma_{i=1}^n \big\| \be_i^\top \big[ \diag(\Gram_{\bW}) \big]^{-1} \Gram_{\bW} (\bI_n - \bH_i) \big\|_2^2 }, \label{eq:var.estimator.hd.partial} 
\end{equation}
where the $i$-th entry of $\tbvarepsilon_{\emph{partial}}$ is given by $\tbvarepsilon_i$ as defined in Corollary \ref{cor:loo.residual}. 
If Assumption~\ref{assump:lm} holds, then we have 
\begin{equation}
        \bbE[\hsigmapin^2] = \sigma^2 + \frac{ \| \bbE[\tbvarepsilon_{\emph{partial}}]\|_2^2}{\Sigma_{i=1}^n \big\| \be_i^\top \big[ \diag(\Gram_{\bW}) \big]^{-1}\Gram_{\bW} (\bI_n - \bH_i) \big\|_2^2}. \label{eq:bias.hd.partial} 
    \end{equation} 
\end{theorem} 
Theorem~\ref{thm:var.estimator.pr} states that $\hsigma^2_{\text{partial}}$, as with $\hsigma^2_{\text{full}}$, is biased. 

\subsection{Variance Estimation via the In-Sample Residuals of Partial Regularization}\label{sec:ve_partial}
In the classical regime, \citet{ding_fwl} showed that the in-sample residuals induced by partial regression can be used to construct variance estimators for $\sigma^2$. 
Motivated by this idea, we artificially construct nontrivial in-sample residuals by transforming the response and one sub-covariate matrix to a subspace related to the other sub-covariate matrix, according to the formulations implied by $\hbbeta^{[\cW]}_{\cW}$ and $\hbbeta^{[\cW]}_{\cW^c}$ in \eqref{eq:beta.fwl.hd}. 
This yields the following two estimators. 

\paragraph{Variance estimation via $\hbbeta^{[\cW]}_{\cW}$.} 
By \eqref{eq:beta.fwl.hd}, $\hbbeta^{[\cW]}_{\cW}$ is obtained by regressing $\bP_{\bT}^\perp \by$ on $\bP_{\bT}^\perp \bW$. 
Notably, the in-sample residuals induced by this formulation are trivially zero. 
To see this, observe $\bP_{\bT}^\perp \by -  \bP_{\bT}^\perp \bW \hbbeta^{[\cW]}_{\cW} = \bzero$. 
Accordingly, we rewrite $\hbbeta^{[\cW]}_{\cW}$ as $\hbbeta^{[\cW]}_{\cW} = (\bP_{\bT}^\perp \bW)^\dagger \by$,
which can be obtained by regressing $\by$ on $\bP_{\bT}^\perp \bW$. 
We utilize this form since these in-sample residuals, $\by -  \bP_{\bT}^\perp \bW \hbbeta^{[\cW]}_{\cW}$, are nontrivial. 

\paragraph{Variance estimation via $\hbbeta^{[\cW]}_{\cW^c}$.} 
Following the arguments above, we can also construct a variance estimator based on the in-sample residuals induced by $\hbbeta^{[\cW]}_{\cW^c}$, which \eqref{eq:beta.fwl.hd} reveals can be viewed as the OLS solution obtained by regressing $\bW^\dagger \by$ on $\bW^\dagger \bT$. 

\paragraph{Formal results.} We formally present two variance estimators based on the discussions above and analyze their biases below. 
\begin{theorem} \label{thm:var.estimator.partial} 
Let Assumption~\ref{assump:partial} hold. 
Define 
	\begin{align}
		\hsigma^2_{\cW} &= \frac{ \| \by - \bP_{\bT}^\perp \bW \cdot \hbbeta^{[\cW]}_{\cW} \|_2^2}{\rank{(\bT)}}. \label{eq:var.estimator.partial.hd.j} 
	\end{align} 
	If Assumption~\ref{assump:lm} holds, then we have  
	\begin{align}
    		\bbE[\hsigma^2_{\cW}] &= \sigma^2 + \frac{ \| \bP_{\bT}\bbE[\by] \|_2^2}{\rank{(\bT)}}. \label{eq:bias.j.hd} 
    	\end{align} 
	Next, define 
	\begin{align}
		\hsigma^2_{\cW^c} = \frac{ \| \bW^\dagger \by - \bW^\dagger \bT \cdot \hbbeta^{[\cW]}_{\cW^c} \|_2^2}{\tr(\bP_{\bT}^\perp \Gram_{\bW})}. \label{eq:var.estimator.partial.hd.jc} 		
	\end{align}
	If Assumption~\ref{assump:lm} holds, then we have 
	\begin{align}
		\bbE[\hsigma^2_{\cW^c}] &= \sigma^2 + \frac{ \|  \bW^\dagger \bP_{\bT}^\perp \bbE[\by] \|_2^2}{\tr(\bP_{\bT}^\perp \Gram_{\bW})}. \label{eq:bias.jc.hd} 
    	\end{align}  
\end{theorem} 

Similar to our LOO-based homoskedastic variance estimator outlined in \eqref{eq:var.estimator.hd.partial}, the two estimators defined in \eqref{eq:var.estimator.partial.hd.j} and \eqref{eq:var.estimator.partial.hd.jc} are also conservative. 
We leave a formal analysis of all three estimators under specific generative models for $\bX = [\bW, \bT]$ as important future work.

\section{Simulation Studies}\label{sec:simulation}

In this section, we evaluate the variance estimators from Section \ref{sec:stat_infer} through four simulations, examining their finite-sample properties from various perspectives and under different generative models.

\subsection{Covariate Models} \label{sec:sim.covariates} 
In our construction of $\bX = [\bW, \bT]$, we set $\bT = \bone$. 
We generate $\bW \in \Rb^{n \times q}$ with $q = p-1$ via the three different processes below with Assumption~\ref{assump:partial} upheld. 
\begin{enumerate}[label=(\alph*)]

\item \emph{Standard normal model.} 
Let the entries of $\bW$ be i.i.d. draws from a standard normal. 

\item \emph{Spiked model} \citep{spiked_model}. 
Let $\bW = \bU \bSigma^{1/2}$, where $\bU \in \Rb^{n \times q}$ is a random matrix with orthonormal rows and $\bSigma = \sigma_x^2 \cdot \left( \bI_{q} + \sum_{\ell=1}^k \lambda_\ell \bv_\ell \bv_\ell^\top \right) \in \Rb^{q \times q}$ with $\sigma_x \in \Rb$, $\lambda_\ell \gg 1$, and $v_\ell$ sampled uniformly at random from the unit sphere of dimension $q-1$. 

\item \emph{Geometric model.} 
Let $\bW = \bU \bSigma^{1/2} \bV^\top$, where $\bU \in \Rb^{n \times q}$ and $\bV \in \Rb^{q \times q}$ are random matrices with orthonormal rows and $\bSigma = \lambda^2 \diag\left( \rho^\ell \right) \in \Rb^{q \times q}$ with $\lambda \in \Rb$, $\rho \in (0,1)$. 

\end{enumerate}

\subsection{Experiment Design}

We will analyze the biases $\hsigma^2 - \sigma^2$ for each variance estimator and each covariate generative model. While Assumption~\ref{assump:lm} assumes that $\bX = [\bW,\bone]$ is fixed, our simulations introduce randomness in $\bW$ to evaluate the average performance of the estimators. Specifically, each bias value is averaged over 100 trials, where in each trial, we make 100 observations $\by$ with a fixed independent draw of $\bW$, resulting in a total of 10,000 simulation repeats.

\begin{remark}[Omitting $\hsigmareg^2$]
    In this experimental setting, the bias of $\hsigmareg^2$ is given by $n^{-1}(\sum_{i=1}^n \bw_i^\top\bbeta_1 + n\beta_0)^2$. Consequently, when $\bW$ is generated using the standard normal model, we have $\text{bias}(\hsigmareg^2) = O(n)$, resulting in a significantly larger bias compared to the other three estimators. Therefore, we omit the results for $\hsigmareg^2$ in the main section and refer readers to the supplementary materials for the full results.
\end{remark}

\subsection{Simulation I: Fixed $p$ and Varying $n$} \label{sec:sim.bias.1} 
%

\paragraph{Data generating process.} \label{sec.sim.1.dgp} 
We fix $p = 100$ and set $\bbeta = (\bbeta_1^\top,\beta_0)^\top$ with $\beta_0=1$ and $\bbeta_1 = p^{-1/2} \bone$.
We vary the sample size $n \in \{20, 40, 60, 80, 99\}$ to ensure that $\bW$ is full row rank.
For each $n$, we generate three covariate matrices $\bW$ as described in Section~\ref{sec:sim.covariates}.
For the spiked model, we choose $\sigma^2_x = 1$, $k = 5$, and sample $\lambda_\ell$ independently from a uniform distribution over $[10, 20]$. For the geometric model, we set $\lambda = 1$ and $\rho = 0.95$.
We construct the response vector $\by = \bX \bbeta + \bvarepsilon = \bW \bbeta_1+\beta_0\mathbf{1}  + \bvarepsilon$, where the entries of $\bvarepsilon$ are sampled from $\Normal(\mathbf{0}, \bI_n)$, i.e., $\sigma^2 = 1$.

\paragraph{Simulation results.} \label{sec:sim.1.results} 

As shown in Figure \ref{fig:bias_sim1}, the finite-sample bias for all three estimators is more pronounced under the standard normal model, while they exhibit near-zero bias in the spiked and geometric models.

\subsection{Simulation II: Fixed Ratio of $n/p$ and Increasing Dimension $p$} \label{sec:sim.bias.2} 
\paragraph{Data generating process.} \label{sec:sim.2.dgp}
Consider a fixed ratio of $n/p = 0.8$. 
For each $p \in \{50,75,100, 125,150\}$, we set $\bbeta = (\bbeta_1^\top,\beta_0)^\top$, where $\beta_0 =1$ and $\bbeta_1 = p^{-1/2} \bone$. 
Then we generate covariate matrix with the same parameters as in Section~\ref{sec.sim.1.dgp}. 
We generate the response $\by = \bW \bbeta_1+\beta_0\mathbf{1}  + \bvarepsilon$ by sampling $\bvarepsilon \sim \Normal(\bzero, \bI_n)$.

\paragraph{Simulation results.} \label{sec:sim.2.results} 

Figure \ref{fig:bias_sim2} reveals a similar observation to Simulation I: the average biases of the three estimators remain near zero across different covariate sizes in the spiked and geometric models, but are higher in the standard normal model.

\subsection{Simulation III: Fixed $(n, p)$ with Increasing Noise Variance $\sigma^2$}  \label{sec:sim.bias.3} 

\paragraph{Data generating process.} 
We maintain a fixed aspect ratio of $n/p = 0.8$ with $p = 100$ and $n = 80$.
As before, we set $\bbeta = (\bbeta_1^\top, \beta_0)^\top$, where $\beta_0 = 1$ and $\bbeta_1 = p^{-1/2} \bone$, and generate the covariate matrix $\bW$ using the same parameters as in Section~\ref{sec.sim.1.dgp}.
We consider increasing noise levels $\sigma \in \{1, 2, 5, 7, 10\}$.
For each $\sigma$, we generate $\by = \bW \bbeta + \beta_0 \mathbf{1}_n + \bvarepsilon$ by sampling $\bvarepsilon \sim \Normal(\bzero, \sigma^2\bI_n)$.

\paragraph{Simulation results.} Figure \ref{fig:bias_sim3} illustrates that the three estimators exhibit near-zero bias across all models, with differences in their standard deviation.

\subsection{Simulation IV: fixed $(n, p)$ with increasing intercept magnitude}  \label{sec:sim.bias.4} 
\paragraph{Data generating process.} 
In this experiment, we aim to investigate how un-regularized columns affect the performance of different variance estimators. 
We maintain a fixed aspect ratio of $n/p = 0.8$ with $p = 100$ and $n = 80$.
We consider $\beta_0 \in \{1, 2, 5, 7, 10\}$ and set $\bbeta = (\bbeta_1^\top,\beta_0)^\top$, where $\bbeta_1 = p^{-1/2} \bone$. We generate the covariate matrix $\bW$ using the same parameters as in Section~\ref{sec.sim.1.dgp}.
For each $\beta_0$, we generate the response vector $\by = \bW \bbeta + \beta_0\mathbf{1} + \bvarepsilon$, where $\bvarepsilon \sim \Normal(\mathbf{0}, \bI_n)$.
%
\paragraph{Simulation results.} 
As shown in Figure~\ref{fig:bias_sim4}, in the standard normal model, unlike the previous experiments where the three estimators behaved similarly, the bias of $\hsigmaf$ increases significantly as the intercept magnitude grows, while $\hsigmaunreg^2$ and $\hsigmap^2$ maintain near-zero bias across different intercept values.
In both the spiked and geometric models, the average biases of all three estimators remain nearly zero, with their standard deviations relatively constant.

\begin{remark} [General takeaways] \label{rmk:simu_takeaway}
    Based on the observations above, $\hsigmaunreg^2$ consistently exhibits low bias and variance, making it the most practical variance estimator. Although $\hsigmapin$ also demonstrates similar properties, it is more computationally demanding due to the complexity of the LOO residual formulas. On the other hand, $\hsigmafin$ can show significant bias in certain scenarios while $\hsigmareg^2$ has a bias of order $n$. 
    Additionally, our results suggest that the estimators perform best when the covariate matrix has approximately low-rank structure (e.g., spiked and geometric models). 
    This finding aligns itself with the conclusions drawn in seminal works a la \citet{bartlett2020benign} that argue the OLS interpolator demonstrates benign overfitting when the covariates are approximately low-rank. 
\end{remark}

\section{Conclusion}\label{sec:conclusion}

This article explored fundamental algebraic and statistical properties of the partially regularized OLS interpolator. We derived high-dimensional partially regularized counterparts of Cochran's formula and the leave-one-out (LOO) formula, both of which are agnostic to statistical modeling assumptions.
Additionally, we extended homoskedastic variance estimation under the Gauss-Markov model to the high-dimensional partial regularization setting, providing a basis for conducting statistical inference, albeit with slight conservatism.
We believe our work enhances the understanding of overparameterized linear regression and complements ongoing research on benign overfitting. 

There are several avenues of future research. 
In view of the work of \citet{cinelli2020making} that reparameterized Cochran's formula in terms of partial $R^2$ values in the classical regime to allow for practical sensitivity analysis, it would be interesting to develop an analogous result in the high-dimensional regime, which inevitably requires a new notion of a high-dimensional $R^2$. 
Moreover, extending the Gauss-Markov theorem to potentially establshing the optimality of the partially regularized OLS interpolator may be of interest. 
Further, as previously mentioned, formal analyses of the ATE estimation problem in Section~\ref{sec:intro.moti} and variance estimators proposed in Section~\ref{sec:stat_infer} under different DGPs could have substantive impacts across a myriad of problems in statistics and machine learning. 

\begin{figure*}[h!]
	\centering  
        \begin{subfigure}{0.29\linewidth}
		\centering 
		\includegraphics[width=\textwidth]
		{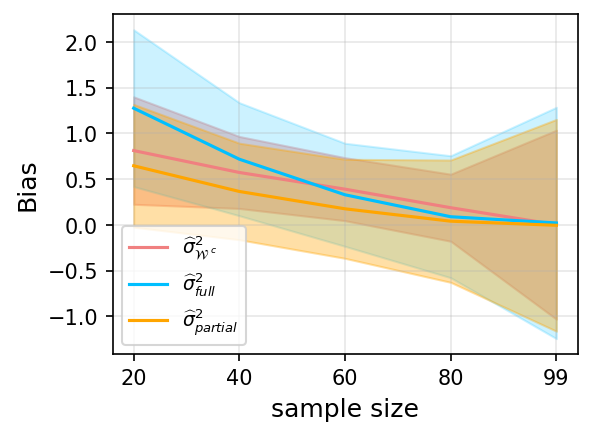}
		\caption{Standard normal model.} 
		\label{fig:standard_normal}
	\end{subfigure} 
        \begin{subfigure}{0.29\linewidth}
		\centering 
		\includegraphics[width=\textwidth]
		{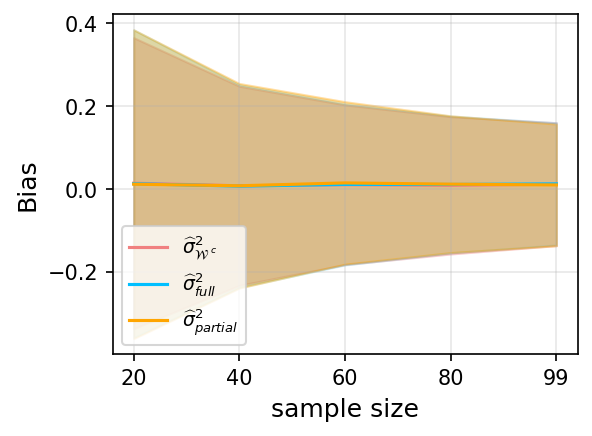}
		\caption{Spiked model.} 
		\label{fig:spiked} 
	\end{subfigure}
        \begin{subfigure}{0.29\linewidth}
		\centering 
		\includegraphics[width=\textwidth]
		{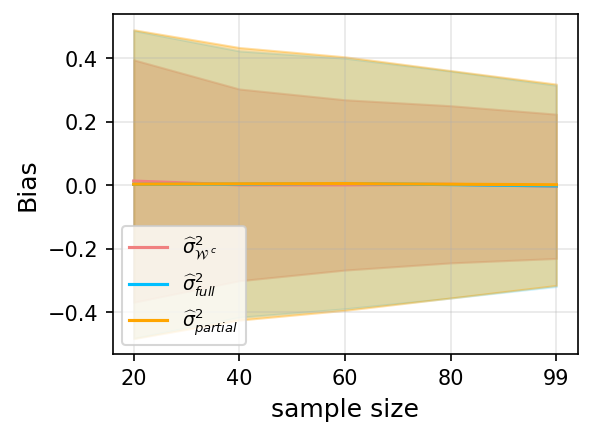}
		\caption{Geometric model.} 
		\label{fig:geometric} 
	\end{subfigure}
	\caption{Simulation results with fixed $p = 100$ and varying $n \in \{20,40,60,80,99\}$. 
	The solid lines represent the average bias over 100 trials, with shading indicating $\pm$ one standard error.}
	\label{fig:bias_sim1} 
\end{figure*}

\begin{figure*}[h!]
	\centering 
	%
        \begin{subfigure}{0.29\linewidth}
		\centering 
		\includegraphics[width=\textwidth]
		{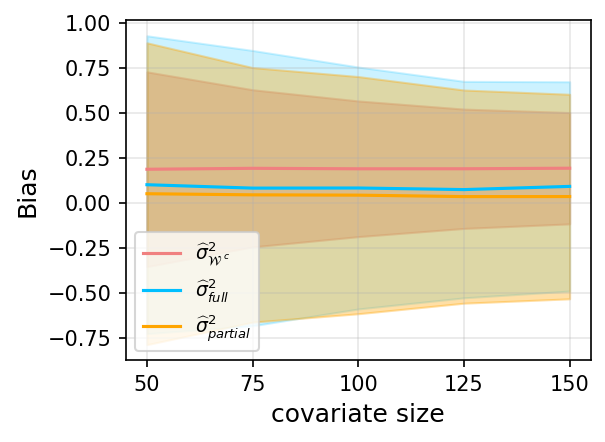}
		\caption{Standard normal model.} 
		\label{fig:standard_normal.2}
	\end{subfigure} 
        \begin{subfigure}{0.29\linewidth}
		\centering 
		\includegraphics[width=\textwidth]
		{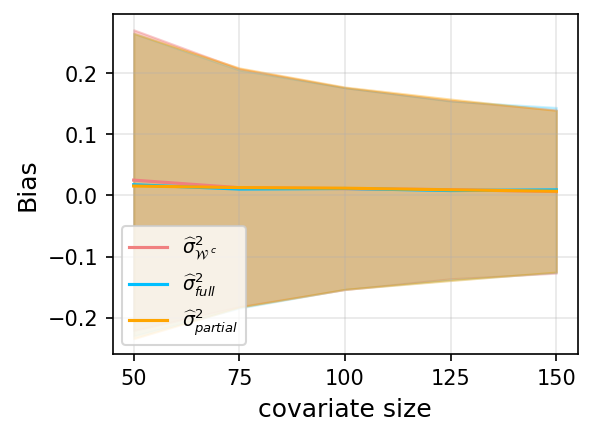}
		\caption{Spiked model.} 
		\label{fig:spiked.2} 
	\end{subfigure}
        \begin{subfigure}{0.29\linewidth}
		\centering 
		\includegraphics[width=\textwidth]
		{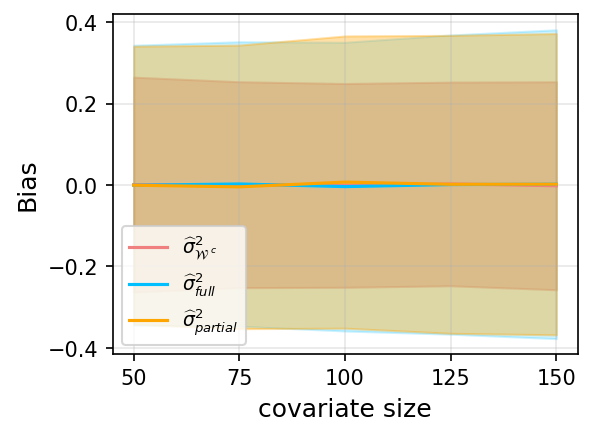}
		\caption{Geometric model.} 
		\label{fig:geometric.2} 
	\end{subfigure}
	\caption{Simulation results with fixed ratio $n/p=0.8$ and varying $p\in\{50,75,100,125,150\}$. 
	The solid lines represent the average bias, with shading indicating $\pm$ one standard error.} 
	\label{fig:bias_sim2} 
\end{figure*}

\begin{figure*}[h!]
	\centering 
        \begin{subfigure}{0.29\linewidth}
		\centering 
		\includegraphics[width=\textwidth]
		{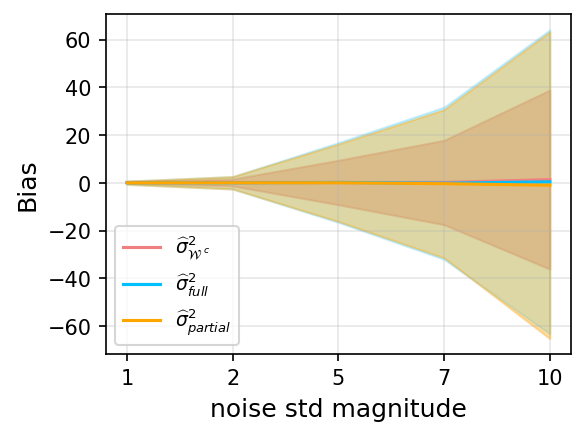}
		\caption{Standard normal model.} 
		\label{fig:standard_normal.3}
	\end{subfigure} 
        \begin{subfigure}{0.29\linewidth}
		\centering 
		\includegraphics[width=\textwidth]
		{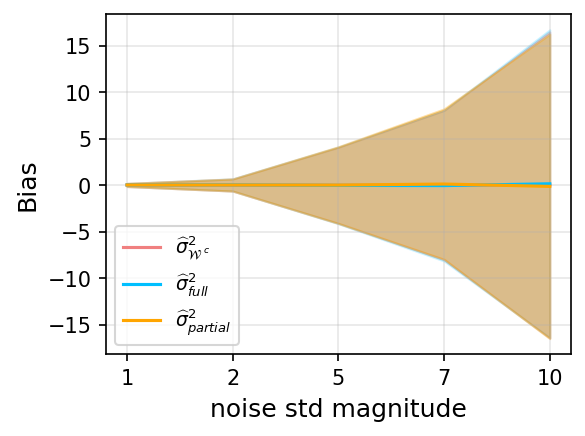}
		\caption{Spiked model.} 
		\label{fig:spiked.3} 
	\end{subfigure}
        \begin{subfigure}{0.29\linewidth}
		\centering 
		\includegraphics[width=\textwidth]
		{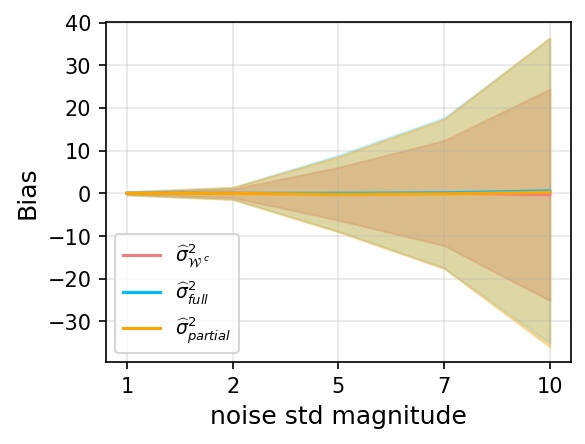}
		\caption{Geometric model.} 
		\label{fig:geometric.3} 
	\end{subfigure}
	\caption{Simulation results with fixed $p=100$, $n=80$, and varying $\sigma \in \{1, 2, 5, 7, 10 \}$. 
	The solid lines represent the average bias, with shading indicating $\pm$ one standard error.}
	\label{fig:bias_sim3} 
\end{figure*}

\begin{figure*}[h!]
	\centering 
        \begin{subfigure}{0.27\linewidth}
		\centering 
		\includegraphics[width=\textwidth]
		{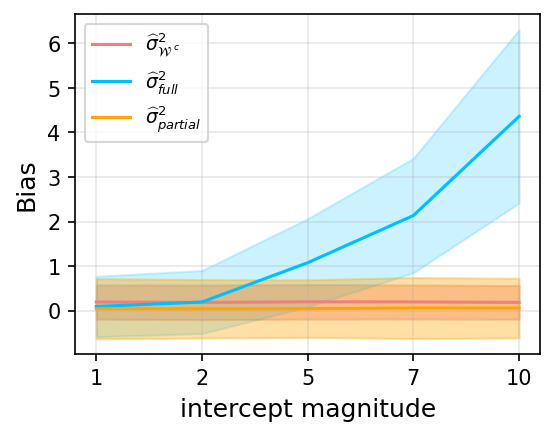}
		\caption{Standard normal model.} 
		\label{fig:standard_normal.4}
	\end{subfigure} 
        \begin{subfigure}{0.29\linewidth}
		\centering 
		\includegraphics[width=\textwidth]
		{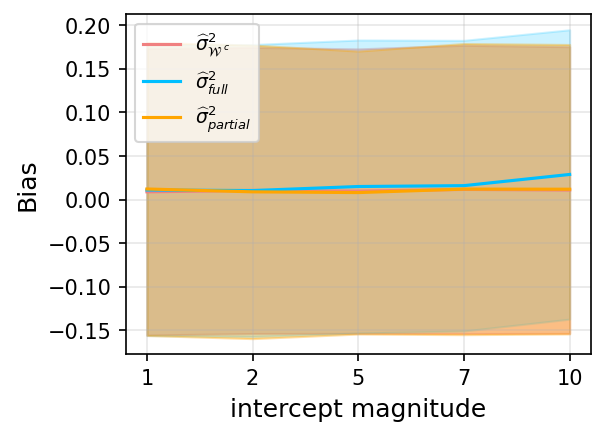}
		\caption{Spiked model.} 
		\label{fig:spiked.4} 
	\end{subfigure}
        \begin{subfigure}{0.29\linewidth}
		\centering 
		\includegraphics[width=\textwidth]
		{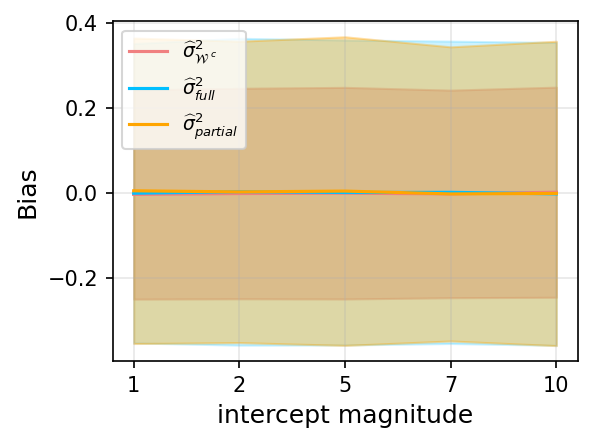}
		\caption{Geometric model.} 
		\label{fig:geometric.4} 
	\end{subfigure}
	%
	\caption{Simulation results with fixed $p=100$, $n=80$, and varying intercept magnitude $\beta_0 \in \{1, 2, 5, 7, 10 \}$. 
	The solid lines represent the average bias, with shading indicating $\pm$ one standard error.}
	\label{fig:bias_sim4} 
\end{figure*}

\bibliographystyle{apalike}
\bibliography{ref.bib}

\appendix
\onecolumn
\newpage
\tmptitle{Supplementary Materials}
\renewcommand{\theequation}{S.\arabic{equation}}
\setcounter{equation}{0}
\renewcommand{\thesection}{S.\arabic{section}}
\setcounter{section}{0}
\renewcommand{\thefigure}{S.\arabic{figure}}
\setcounter{figure}{0}
Appendix \ref{apdx:notation} recollects the notations used in the main text. Appendix \ref{apdx:pinv} provides an overview of the properties of the Moore-Penrose pseudoinverse.
Appendix \ref{apdx:intro.moti} provides the details of the simulation experiment from Section \ref{sec:intro.moti}.
Appendices \ref{apdx:col_par}, \ref{apdx:row_par}, and \ref{apdx:stat_infer} contain the proofs of the theorems, propositions, and corollaries deferred from Sections \ref{sec:col_par}, \ref{sec:row_par}, and \ref{sec:stat_infer}, respectively.
The code for implementing the simulation is provided in the attached zip file.

\section{Recollecting notation}\label{apdx:notation}

\subsection{General}

For any $n \in \NN$, let $[n] \coloneqq \{ 1, \dots, n \}$.
We denote $\langle \bx, \by \rangle = \sum_{i=1}^n x_i y_i$ for $\bx, \by \in \RR^n$, and $\| \bv \|_2 \coloneqq \langle \bv, \bv \rangle^{1/2}$. 
We reserve $\be_i$ to signify the $i$-th standard basis vector of $\RR^n$ for each $i \in [n]$. 
For any vector subspace $\cV \subseteq \RR^n$, let $\cV^{\perp} \coloneqq \{ \bw \in \RR^n: \langle \bw, \bv \rangle = 0, ~\forall \bv \in \cV \}$ represent the orthogonal complement of $\cV$ in $\RR^n$. 
By default, a vector in $\RR^n$ is assumed to be in its column representation of size $n \times 1$, unless stated otherwise. 
For any $n \in \NN$, let $\bI_n$ denote the $n \times n$ identity matrix.

\subsection{Matrix}

For any matrix $\bM \in \Rb^{n \times p}$, let $\| \bM \|_{\F} \coloneqq \tr(\bM^{\top} \bM)^{1/2}$ denote its Frobenius norm. 
The transpose, inverse (if invertible), and Moore-Penrose pseudoinverse of $\bM$ are denoted by $\bM^{\top}$, $\bM^{-1}$, and $\bM^{\dagger}$, respectively.
The column space and row space of $\bM$ are represented as $\colsp(\bM) \subseteq \RR^n$ and $\rowsp(\bM) \subseteq \RR^p$, respectively.
Let $\bP_{\bM} \coloneqq \bM \bM^{\dagger}$ denote the projection matrix onto the column space of $\bM$. We define $\bP_{\bM}^{\perp} = \bI_n - \bP_{\bM}$ as the projection matrix onto the orthogonal complement of $\colsp(\bM)$.
Moreover, let $\Gram_{\bM} = (\bM \bM^\top)^{\dagger} \in \Rb^{n \times n}$ denote the ``inverse Gram matrix'' for the rows of $\bM$.

For nonempty sets $\cI \subseteq [n]$ and $\cJ \subseteq [p]$, let $\bM_{\cI, \cJ} \in \RR^{|\cI| \times |\cJ|}$ denote the submatrix of $\bM$ containing elements $M_{ij}$ for $(i,j) \in \cI \times \cJ$. In cases where $\cI = [n]$, we use $\bM_{\star, \cJ}$ to represent $\bM_{\cI, \cJ}$; similarly, if $\cJ = [p]$, we denote $\bM_{\cI, \star}$ instead of $\bM_{\cI, \cJ}$. For vectors, $\bv_{\cI}$ refers to the subvector of $\bv$ containing the coordinates indexed by $\cI$. 
In our notation, transpose and pseudoinverse operations take precedence over submatrix operations, e.g., $\bM^{\dagger}_{\cJ, \cI}$ refers to the submatrix of $\bM^{\dagger}$, whereas $(\bM_{\cI, \cJ})^{\dagger}$ denotes the pseudoinverse of the submatrix $\bM_{\cI, \cJ}$.

We use the notation $\diag$ for the map that produce diagonal matrices, i.e., for $\bM \in \Rb^{n \times n}$, let $\diag(\bM) \coloneqq \sum_{i=1}^n M_{ii} \be_i \be_i^\top$.
\section{The Moore-Penrose pseudoinverse} \label{apdx:pinv} 
This section provides a brief technical background of the Moore-Penrose pseudoinverse, beginning with its formal definition below. 
\begin{definition}\label{defn:pseudo}
    For $\bM \in \RR^{m \times n}$, the pseudoinverse of $\bM$, denoted by $\bM^{\dagger} \in \RR^{n \times m}$, is defined as a matrix that satisfies the following four conditions, known as the Moore-Penrose criteria:
    \begin{enumerate}
        \item 
        $\bM \bM^{\dagger} \bM = \bM$; 
        \item 
        $\bM^{\dagger} \bM \bM^{\dagger} = \bM^{\dagger}$; 
        \item
        $(\bM \bM^{\dagger})^{\top} = \bM \bM^{\dagger}$;
        \item 
        $(\bM^{\dagger} \bM)^{\top} = \bM^{\dagger} \bM$.
    \end{enumerate}
\end{definition} 

For any matrix $\bM$, the pseudoinverse $\bM^{\dagger}$ exists and is unique \citep{golub2013matrix}; that is, there is precisely one matrix $\bM^{\dagger}$ that satisfies the four properties outlined in Definition \ref{defn:pseudo}. 
Moreover, it can be computed using the singular value decomposition (SVD): if $\bM = \bU \bS \bV^{\top}$ is the compact SVD of $\bM$, where $\bS$ is a square diagonal matrix of size $r \times r$ with $r = \rank(\bM) \leq \min\{m,n\}$, the pseudoinverse is given by $\bM^{\dagger} = \bV \bS^{-1} \bU^{\top}$.

\paragraph{Basic properties.} 
Here we review some useful properties of the pseudoinverse.
\begin{enumerate}[label=(\alph*)]
    \item 
    If $\bM$ is invertible, its pseudoinverse is its inverse, i.e., $\bM^{\dagger} = \bM^{-1}$.
    \item 
    The pseudoinverse of the pseudoinverse is the original matrix, i.e., $(\bM^{\dagger})^{\dagger} = \bM$.
    \item 
    Pseudoinverse commutes with transposition: $(\bM^{\top})^{\dagger} = (\bM^{\dagger})^{\top}$. Thus, we may use notations such as $\bM^{\dagger, \top}$ or $\bM^{\top, \dagger}$ exchangeably, omitting parentheses for notational brevity.
    \item 
    The pseudoinverse of a scalar multiple of $\bM$ is the reciprocal multiple of $\bM^{\dagger}$: $(\alpha \bM)^{\dagger} = \alpha^{-1} \bM^{\dagger}$ for $\alpha \neq 0$.
    \item
    The computation of the pseudoinverse can be expressed in terms of the pseudoinverses of the symmetric matrices $\bM^\top \bM$ and $\bM \bM^\top$: 
    \begin{align}
         \bM^\dagger = (\bM^\top \bM)^\dagger \bM^\top=\bM^\top (\bM \bM^\top)^\dagger. \label{eq:pseudo_compute}
    \end{align}
    \item 
    If a column-wise partitioned block matrix $\bM = \begin{bmatrix} \bA & \bB\end{bmatrix}$ is of full column rank (has linearly independent columns), then 
    \begin{align}
        \bM^{\dagger}
            &=\begin{bmatrix} \bA & \bB \end{bmatrix}^{\dagger}
            = \begin{bmatrix}   (\bP_{\bB}^{\perp}\bA)^{\dagger} \\ (\bP_{\bA}^{\perp}\bB)^{\dagger} \end{bmatrix},
                \nonumber\\
        \bM^{\top, \dagger}
            &= \begin{bmatrix} \bA^{\top} \\ \bB^{\top} \end{bmatrix}^{\dagger}
            = \begin{bmatrix}   (\bA^{\top}\bP_{\bB}^{\perp})^{\dagger} & (\bB^{\top}\bP_{\bA}^{\perp})^{\dagger} \end{bmatrix}.
                \label{eqn:block_pseudo}
    \end{align}
\end{enumerate}

Recall the property $(\bA \bB )^{-1} = \bB^{-1} \bA^{-1}$, which holds for invertible matrices $\bA$ and $\bB$. A similar correspondence exists for the pseudoinverse under certain conditions, although it is not universally applicable. Here, we present several equivalent conditions for this relationship, outlined in \cite{greville1966note}, as a lemma.

\begin{lemma}
    Let $\bA, \bB$ be real matrices. The following are equivalent:
    \begin{enumerate}[label=(\roman*)]
        \item 
        $(\bA \bB)^{\dagger} = \bB^{\dagger} \bA^{\dagger}$.
        \item 
        $\bA^{\dagger}\bA\bB\bB^{\top}\bA^{\top} = \bB\bB^{\top}\bA^{\top}$ and $\bB\bB^{\dagger}\bA^{\top}\bA\bB = \bA^{\top}\bA\bB$.
        \item
        $\bA^{\dagger}\bA\bB\bB^{\top}$ and $\bA^{\top}\bA\bB\bB^{\dagger}$ are both symmetric.
        \item 
        $\bA^{\dagger}\bA\bB\bB^{\top}\bA^{\top}\bA\bB\bB^{\dagger} = \bB\bB^{\top}\bA^{\top}\bA$.
        \item 
        $\bA^{\dagger}\bA\bB = \bB(\bA\bB)^{\dagger}\bA\bB$ and $\bB\bB^{\dagger}\bA^{\top} = \bA^{\top} \bA\bB (\bA\bB)^{\dagger}$.
    \end{enumerate}
\end{lemma}

\section{DGP for motivation example in Section \ref{sec:intro.moti}}\label{apdx:intro.moti}

In Section \ref{sec:intro.moti}, we compared the performance of the fully regularized and partially regularized OLS interpolators under a classical causal inference framework, extended to the high-dimensional setting.
Recall that we regress $\by$ on $[\bW, \bD, \bone]$, with the goal of estimating the average treatment effect (ATE) through the coefficient associated with $\bD$.
Specifically, we construct the response vector as $\by = \bW \balpha + \tau \bD + \alpha_0 \bone + \bvarepsilon$.
In the simulation, we fix $n = 80$ and $p = 100$, $\balpha = p^{-\frac{1}{2}}\bone$, $\alpha_0=1$, and set $\tau = [0, \pm 1, \pm 2, \pm 4, \pm 6, \pm 8]$. 
For each $\tau$ value, we generate $\bW$ from a spiked covariance model as described in Section \ref{sec:sim.covariates}, with $k = 5$, $\sigma_x^2 = 1$, and $\lambda_l$ independently sampled from a uniform distribution over $[10,20]$. The treatment variable $\bD$ is generated as $\bD \sim \text{Bernoulli}(0.5)^n$, and the noise term $\bvarepsilon \sim \Normal(\bzero, \bI)$.
To evaluate the performance of the estimators, we compute the average bias over 100 trials, with each trial consisting of 100 different observations of $\by$, using a fixed independent draw of the covariate matrix $\bW$ and treatment assignment $\bD$. This results in a total of 10,000 repetitions for each bias estimate.

\section{Deferred proof from section \ref{sec:col_par}}\label{apdx:col_par}

Before formally proving Theorem \ref{thm:cochran's}, we first make the following preparations:

\textbf{Preparation 1.} We rewrite $\hbbeta^{[\cW]}_{\cW}$ \eqref{eq:beta.fwl.hd} in an alternative form as:

\begin{align}
\hbbeta^{[\cW]}_{\cW} &=
\bP_{\bW^\top} \bP_{\bW^\dagger \bT}^\perp \bW^\dagger \by \nonumber\\
&=
\bW^{\top} \bW^{\top, \dagger}(\bI_q-\bP_{\bW^\dagger \bT}) \bW^\dagger \by
\quad (\because \text{def. of projection matrix})
\nonumber\\
&=
\bW^{\top} \bW^{\top,\dagger} \bW^{\dagger} \by - \bW^{\top} \bW^{\top,\dagger} \bP_{\bW^{\dagger} \bT} \bW^{\dagger} \by 
\nonumber\\
&=
\bW^{\top} \bG_{\bW} \by - \bW^{\top} \bW^{\top, \dagger}(\bW^\dagger \bT)(\bW^\dagger \bT)^{\dagger} \bW^{\dagger} \by 
\quad (\because \Gram_{\bW} = (\bW \bW^\top)^\dagger)
\nonumber\\
&=
\bW^{\top} \bG_{\bW} \by - \bW^{\top} \Gram_{\bW} \bT \hbbeta^{[\cW]}_{\cW^c}
\quad (\because \hbbeta^{[\cW]}_{\cW^c}=(\bW^\dagger \bT)^{\dagger} \bW^{\dagger} \by \text{ by Remark \ref{rmk:fwl.j.hd.1}})
\nonumber\\
&=
\bW^{\top} \bG_{\bW} (\by - \bT \hbbeta^{[\cW]}_{\cW^c}) \label{eq:fwl.j.hd.alt}
\end{align}

\textbf{Preparation 2.} Also, we can rewrite $\hbbeta^{[\cW]}_{\cW^c}$ \eqref{eq:beta.fwl.hd}  as: $
\hbbeta^{[\cW]}_{\cW^c} = \left( \bT^\top \Gram_{\bW} \bT \right)^\dagger \bT^\top \Gram_{\bW} \by
$. To see this, we have:

\begin{align}
    \hbbeta^{[\cW]}_{\cW^c} &= (\bW^\dagger \bT)^\dagger \bW^\dagger \by \nonumber\\
    &= [(\bW^\dagger \bT)^\top \bW^\dagger \bT]^\dagger (\bW^\dagger \bT)^\top \bW^\top \Gram_{\bW} \by 
    \quad \because \eqref{eq:pseudo_compute}
    \nonumber\\
    &= [\bT^\top \Gram_{\bW} \bT]^\dagger \bT^\top \bW^{\dagger,\top} \bW^\top \Gram_{\bW} \by \nonumber\\
    &= [\bT^\top \Gram_{\bW} \bT]^\dagger \bT^\top \Gram_{\bW} \by \quad (\because \text{rank}(\bW)=n, \bW\bW^\dagger=\bI_n)\label{eq:fwl.jc.hd.gls}
\end{align}

\textbf{Preparation 3.} Recall $\bW=[\bZ, \bU]$, with $\bZ \in \Rb^{n \times l}$ full row rank and $\bU \in \Rb^{n \times r}$ full column rank. By the Sherman–Morrison–Woodbury formula \citep{max1950inverting}, we have:
\[
\Gram_{\bW} = \Gram_{\bZ} - \Gram_{\bZ} \bU (\bI + \bU^\top \Gram_{\bZ} \bU)^{-1} \bU^\top \Gram_{\bZ}.
\]
We denote $\bM = \bI + \bU^\top \Gram_{\bZ} \bU$ ($\bM$ is invertible since $\bU$ is full column rank), $\bA = \Gram_{\bZ} \bU(\bI+\bU^\top \Gram_{\bZ} \bU)^{-1} \bU^\top \Gram_{\bZ} = \Gram_{\bZ} \bU \bM^{-1} \bU^\top \Gram_{\bZ}$. We note the following useful equation:
\begin{equation}
    \Gram_{\bW} \bU = (\Gram_{\bZ} - \bA)\bU = \Gram_{\bZ} \bU \bM^{-1} \label{eq:thm1_pf_useful}
\end{equation}
To see this, note that $\bM = \bI + \bU^\top \Gram_{\bZ} \bU$, and since $\bM$ is invertible, we have $\bI = \bM^{-1} + \bM^{-1} \bU^\top \Gram_{\bZ} \bU$. Therefore, $
\Gram_{\bZ} \bU \bM^{-1} = \Gram_{\bZ} \bU (\bI - \bM^{-1} \bU^\top \Gram_{\bZ} \bU) = (\Gram_{\bZ} - \bA) \bU.
$

With the above preparations complete, we are now ready to prove Theorem \ref{thm:cochran's}.

\begin{proof}[Proof of theorem \ref{thm:cochran's}]

To prove \eqref{eq:cochran_image}, note that under Assumption \ref{assump:partial_3}, 

$$\min_{\btheta_1 \in \Rb^p} \| \by - (\bZ \balpha + \bU \bgamma + \bT \btau) \|_2^2 = 0,$$

$$\min_{\btheta_2 \in \Rb^{l+m}}\| \by - (\bZ \balpha + \bT \btau)\|_2^2 = 0,$$

$$\min_{\btheta_3 \in \Rb^{(l+m)\times r}} \| \bU - (\bZ \bDelta + \bT \bdelta) \|_F^2=0,$$
where $\btheta_1 = \begin{pmatrix}
    \balpha \\
    \bgamma \\
    \btau
\end{pmatrix} \in \Rb^p, 
\btheta_2 = \begin{pmatrix}
    \balpha \\
    \btau
\end{pmatrix} \in \Rb^{l+m},
\btheta_3 = \begin{pmatrix}
    \bDelta \\
    \bdelta
\end{pmatrix} \in \Rb^{(l+m)\times r}$. Therefore, $\by = \bZ \hbalpha + \bU \hbgamma + \bT \hbtau$ for $\forall \hbtheta_1 \in \cS_1$, $\by = \bZ \tbalpha + \bT \tbtau$ for $\forall \tbtheta_2 \in \cS_2$, and $\bU = \bZ \hbDelta + \bT \hbdelta$ for $\forall \hbtheta \in \cS_3$. Therefore,

\begin{align}
    \by &= \by 
    \nonumber\\
    \bZ \hbalpha + \bU \hbgamma + \bT \hbtau &= \bZ \tbalpha + \bT \tbtau 
    \nonumber\\
    \bZ (\hbalpha+ \hbDelta \hbgamma) + \bT (\hbtau + \hbdelta\hbgamma) &= \bZ \tbalpha + \bT \tbtau \quad (\because \bU = \bZ \hbDelta + \bT \hbdelta)
\end{align}
Thus we get \eqref{eq:cochran_image}.

Then we prove $\tbtau - \hbtau = \hbdelta \hbgamma$. Suppose $(\hbalpha, \hbgamma, \hbtau)$, $(\tbalpha, \tbtau)$, and $(\hbDelta, \hbdelta)$ are each the partially regularized OLS estimators as defined in \eqref{eq:partial_in_set}.
According to (\ref{eq:fwl.j.hd.alt}) and (\ref{eq:fwl.jc.hd.gls}), we have the following expressions for the coefficient estimators:

$\hbtau = (\bT^\top \Gram_{\bW} \bT)^\dagger(\bT^\top \Gram_{\mathbf{W}} \mathbf{y})$,
$\tbtau = (\bT^\top \Gram_{\bZ} \bT)^\dagger(\bT^\top \Gram_{\bZ} \mathbf{y})$,
$\hbdelta = (\bT^\top \Gram_{\bZ} \bT)^\dagger(\bT^\top \Gram_{\bZ} \bU)$,
$\hbgamma = \bU^{\top} \Gram_{\bW} (\by - \bT \hbtau)$. 
Note that $\bT^\top \Gram_{\bW} \bT$ is full rank, so we can multiply $\bT^\top \Gram_{\bW} \bT$ on both sides of the equation $\tbtau - \hbtau = \hbdelta \hat{\gamma}$. This leads to the equivalent task of proving the following:
\[
\bT^\top \Gram_{\bZ} \by - (\bT^\top \Gram_{\bZ} \bT)(\bT^\top \Gram_{\bW} \bT)^\dagger \bT^\top \Gram_{\bW} \by
=
\bT^\top \Gram_{\bZ} \bU \left[ \bU^\top \Gram_{\bW} \by - \bU^\top \Gram_{\bW} \bT (\bT^\top \Gram_{\bW} \bT)^\dagger \bT^\top \Gram_{\bW} \by \right].
\]

For clarity, we denote $\bQ = (\bT^\top \Gram_{\bW} \bT)^\dagger$. Then we can express the left-hand side as:

$
\text{LHS} = \bT^\top \Gram_{\bZ} \by - \bT^\top \Gram_{\bZ} \bT  \bQ  \bT^\top \Gram_{\bW} \by = \big[\bT^\top \Gram_{\bZ} - \bT^\top (\Gram_{\bW} + \bA) \bT  \bQ  \bT^\top \Gram_{\bW}\big]\by.
$

For the right-hand side, we substitute \eqref{eq:thm1_pf_useful} and obtain:
\begin{align*}
\text{RHS} &= 
\bT^\top \Gram_{\bZ} \bU \bM^{-1} \bU^\top \Gram_{\bZ} \by 
- 
\bT^\top \Gram_{\bZ} \bU \bM^{-1} \bU^\top \Gram_{\bZ} \bT  \bQ  \bT^\top \Gram_{\bW} \by \\
&=
\big[\bT^\top \bA 
- 
\bT^\top \bA \bT  \bQ  \bT^\top \Gram_{\bW} \big]\by.
\end{align*}

Then we have LHS = RHS by noting that
\begin{align*}
    & \bT^\top \Gram_{\bZ} - \bT^\top (\Gram_{\bW} + \bA) \bT  \bQ  \bT^\top \Gram_{\bW} \\
    =&\bT^\top \Gram_{\bZ} - \bT^\top \Gram_{\bW} - \bT^\top \bA \bT  \bQ  \bT^\top \Gram_{\bW} \quad (\because \bT^\top \Gram_{\bW} \bT \bQ = \bI_m)\\
    =&\bT^\top \bA - \bT^\top \bA \bT  \bQ  \bT^\top \Gram_{\bW}
\end{align*}

After completing the proof that $\tbtau - \hbtau = \hbdelta \hbgamma$, and having established \eqref{eq:cochran_image}, we have
$
\bZ \hbDelta \hbgamma = \bZ (\tbalpha - \hbalpha).
$
Since $\bZ$ is full row rank, it follows that $\tbalpha - \hbalpha = \hbDelta \hbgamma$. This concludes the proof of Theorem \ref{thm:cochran's}.

\end{proof}
\section{Deferred proof from section \ref{sec:row_par}}\label{apdx:row_par}

\subsection{Proof of Proposition~\ref{prop:loo}}

Before starting the formal proof, we first explore some interesting properties of the matrices involved:

Recall $\Gram_{\bW} = (\bW \bW^\top)^\dagger$, $\bW\in \Rb^{n \times q}$. 
Let 
\begin{align}
    \bP_i = \frac{\bW^\dagger \be_i \be_i^\top \bW^{\dagger, \top}}{\be_i^\top \Gram_{\bW} \be_i},\quad \bQ_i = \frac{\be_i \be_i^\top \Gram_{\bW}}{\be_i^\top \Gram_{\bW} \be_i}. \label{mat:QiPi}
\end{align}
Then, we have the following claims:

\begin{align}
    &\text{1. } \bQ_i, \bP_i \text{ are idempotent.}\\
    &\text{2. }  (\bI_q - \bP_i) \bW^\dagger=\bW^\dagger(\bI_n - \bQ_i).\label{eq:prep.2}\\
    &\text{3. } \Gram_{\bW}(\bI_n - \bQ_i) = (\bI_n - \bQ_i)^\top\Gram_{\bW}(\bI_n - \bQ_i).\label{eq:prep.3}\\
    &\text{4. } (\bW_{\si})^\dagger   \bW_{\si} = (\bI_q - \bP_i) \bW^\dagger \bW; \label{eq:prep.4}
\end{align}

The first three claims can be verified through straightforward matrix operations. To see the fourth point, we firstly state a helpful classical result on the pseudoinverse of a rank-one perturbation of an arbitrary matrix $\bA \in \Rb^{n \times p}$ from \cite{gen_inverse_meyers}. 
 \begin{lemma} [{\cite[Theorem 6]{gen_inverse_meyers}}] \label{lemma:meyer}
     Let $\bA \in \RR^{n \times p}$, $\bc \in \RR^n$, and $\bd \in \RR^p$. 
     If $\bc \in \colsp(\bA)$, $\bd \in \rowsp(\bA)$, and $1 + \bd^{\top} \bA^\dagger \bc = 0$, then
     \begin{align}
     	(\bA + \bc \bd^{\top})^\dagger = \bA^\dagger - \bk \bk^\dagger \bA^\dagger - \bA^\dagger \bh^{\dagger, \top} \bh^{\top} 
     	+ (\bk^\dagger \bA^\dagger \bh^{\dagger,\top}) \bk \bh^{\top}, \label{eq:meyer.1} 
     \end{align} 
     where $\bk = \bA^\dagger \bc \in \RR^{p}$ and $\bh = \bA^{\dagger, \top} \bd \in \RR^{n}$. 
 \end{lemma} 

Observe that
\begin{equation}\label{eqn:loo.products}
    (\bX_{\si})^{\top} \bX_{\si} = \bX^{\top} \bX - \bx_i \bx_i^{\top}
    \qquad\text{and}\qquad
    (\bX_{\si})^{\top} \by_{\si} = \bX^{\top} \by - y_i \bx_i = \bX^{\top} \big( \bI_n - \be_i \be_i^{\top} \big) \by,
\end{equation}

We apply Lemma \ref{lemma:meyer} with $\bA = \bX^{\top} \bX$, $\bc = \bx_i = \bX^{\top} \be_i$, and $\bd = -\bx_i = - \bX^{\top} \be_i$. Then, 
\[
    \bk = \bA^{\dagger} \bc = \big( \bX^{\top} \bX \big)^{\dagger} \bX^{\top} \be_i = \bX^{\dagger} \be_i
    \qquad\text{and}\qquad
    \bh = \bA^{\dagger,\top} \bd = - \bX^{\dagger} \be_i
\]
because $\bA^{\dagger} = \bX^{\dagger} (\bX^{\top})^{\dagger}$. 
Since $\bv^{\dagger} = \| \bv \|_2^{-2} \cdot \bv^{\top}$ for any vector $\bv$, we observe that 
\begin{align*}
    \bk \bk^{\dagger} \bA^{\dagger} 
         &= \frac{ \bk \bk^{\top} }{ \| \bk \|_2^2 } \bA^{\dagger}  
         = \frac{ \bX^{\dagger} \be_i \be_i^{\top} (\bX^{\dagger})^{\top} }{ \be_i^{\top} (\bX^{\dagger})^{\top} \bX^{\dagger} \be_i } \big( \bX^{\top} \bX \big)^{\dagger}
         = \bX^{\dagger} \frac{ \be_i \be^{\top}_i}{\be^{\top}_i \Gram_{\bX} \be_i } \Gram_{\bX} ( \bX^{\top} )^{\dagger},
         \\
    \bA^{\dagger} \bh^{\dagger, \top} \bh^{\top} 
         &= \bA^{\dagger} \frac{ \bh \bh^{\top} }{ \| \bh \|_2^2 }  
         =  \big( \bX^{\top} \bX \big)^{\dagger} \frac{ \bX^{\dagger} \be_i \be_i^{\top} (\bX^{\dagger})^{\top} }{ \be_i^{\top} (\bX^{\dagger})^{\top} \bX^{\dagger} \be_i } 
         =  \bX^{\dagger} \Gram_{\bX}  \frac{ \be_i \be^{\top}_i}{\be^{\top}_i \Gram_{\bX} \be_i } ( \bX^{\top} )^{\dagger},
         \\
    (\bk^\dagger \bA^\dagger \bh^{\dagger,\top}) \bk \bh^{\top}
         &= \frac{\bk^{\top} \bA^{\dagger} \bh }{ \| \bk\|_2^2 \| \bh \|_2^2 } \bk \bh^{\top}
         = \frac{ \be_i^{\top} (\bX^{\dagger})^{\top} \big( \bX^{\top} \bX )^{\dagger} \bX^{\dagger} \be_i }{ \big( \be_i^{\top} \Gram_{\bX} \be_i \big)^2 }
             \cdot \bX^{\dagger} \be_i \be_i^{\top} (\bX^{\dagger})^{\top}
          = \frac{ \be_i^{\top} {\Gram_{\bX}}^2 \be_i }{ \big( \be_i^{\top} \Gram_{\bX} \be_i \big)^2 }
             \cdot \bX^{\dagger} \be_i \be_i^{\top} (\bX^{\dagger})^{\top}.
\end{align*}

Thereafter, applying Lemma~\ref{lemma:meyer}, we obtain that
\begin{align}
    \big[ (\bX_{\si})^{\top} \bX_{\si} \big]^{\dagger} 
    &= \big(\bX^{\top} \bX - \bx_i \bx_i^{\top}\big)^{\dagger} 
    \nonumber \\
    &= \big( \bX^{\top} \bX \big)^{\dagger}
    - \bX^{\dagger} \frac{ \be_i \be^{\top}_i}{\be^{\top}_i \Gram_{\bX} \be_i } \Gram_{\bX} ( \bX^{\top} )^{\dagger}
    - \bX^{\dagger} \Gram_{\bX}  \frac{ \be_i \be^{\top}_i}{\be^{\top}_i \Gram_{\bX} \be_i } ( \bX^{\top} )^{\dagger}
    + \frac{ \be_i^{\top} {\Gram_{\bX}}^2 \be_i }{ \big( \be_i^{\top} \Gram_{\bX} \be_i \big)^2 }
    \cdot \bX^{\dagger} \be_i \be_i^{\top} (\bX^{\dagger})^{\top}
    \nonumber \\
    &= \bX^{\dagger} \left\{ \bI_n -  \frac{ \be_i \be^{\top}_i}{\be^{\top}_i \Gram_{\bX} \be_i } \Gram_{\bX} - \Gram_{\bX}  \frac{ \be_i \be^{\top}_i}{\be^{\top}_i \Gram_{\bX} \be_i } + \frac{ \be_i^{\top} {\Gram_{\bX}}^2 \be_i }{ \big( \be_i^{\top} \Gram_{\bX} \be_i \big)^2 } \be_i \be_i^{\top} \right\} (\bX^{\dagger})^{\top}.
    \label{eq:loo.pseudo}
\end{align}

By applying the above result \eqref{eq:loo.pseudo} on $\bW_{\si}$, we get $[(\bW_{\si})^\top \bW_{\si}]^\dagger = \bW^\dagger \left\{\bI_n - \bO_i\right\}(\bW^\dagger)^\top$, where $\bO_i = \frac{ \be_i \be^{\top}_i}{\be^{\top}_i \Gram_{\bW} \be_i } \Gram_{\bW} + \Gram_{\bW}  \frac{ \be_i \be^{\top}_i}{\be^{\top}_i \Gram_{\bW} \be_i } - \frac{ \be_i^{\top} {\Gram_{\bW}}^2 \be_i }{ \big( \be_i^{\top} \Gram_{\bW} \be_i \big)^2 } \be_i \be_i^{\top}$. Then we have

\begin{align*}
    (\bW_{\si})^\dagger \bW_{\si} &= \big[(\bW_{\si})^\top \bW_{\si}\big]^\dagger (\bW_{\si})^\top \bW_{\si} \\
    &=\bW^\dagger \left\{\bI_n - \bO_i\right\}(\bW^\dagger)^\top (\bW^\top \bW - \bw_i \bw_i^\top) \\
    & = \bW^\dagger \left\{\bI_n - \bO_i\right\} \bW - \bW^\dagger \left\{\bI_n - \bO_i\right\}(\bW^\dagger)^\top \bW^\top\be_i \be_i^\top \bW \quad (\because \bW \bW^\dagger = \bI, \bw_i = \bW^\top \be_i) \\
    & = \bW^\dagger \left\{\bI_n - \bO_i\right\} \bW - \bW^\dagger \left\{\bI_n - \bO_i\right\}\be_i \be_i^\top \bW\\
    & = \bW^\dagger \left\{\bI_n - \bO_i\right\} (\bI - \be_i \be_i^\top)\bW \\
    & = (\bI_q - \bP_i) \bW^\dagger \bW
\end{align*}
Thus we finish proving \eqref{eq:prep.4}.

We present the following lemma, which is useful for the proof.

\begin{lemma}[{\cite[Corollary 2]{shen2024}}] \label{lemma:shen_cor2}
    Let $\bX \in \Rb^{n \times p}$, $\by \in \Rb^n$, and $\cI = \{i\}^c$ for any $i \in [n]$. If $\bX$ has full row rank, then
    \begin{align}
        \hbbeta^{(\sim i)}
            &= \left\{ \bI_p - \frac{\bX^{\dagger} \cdot \be_i \be_i^\top \cdot \left(\bX^{\dagger}\right)^\top}{\be_i^\top \cdot \bG_{\bX} \cdot \be_i} \right\} \cdot \hbbeta,  \label{eq:beta.loo.hd}
    \end{align}  
    where $\hbbeta^{(\sim i)}$ is the fully regularized OLS estimator with the $i$-th sample left out, and $\hbbeta = \bX^\dagger \by$ is the full-sample estimator.
\end{lemma}

After deriving these properties, we are now ready to prove Corollary \ref{prop:loo}.

\begin{proof}[Proof of Proposition \ref{prop:loo}]
     Suppose 
     \begin{align}
         \bT = \bW \hbphi, \bT_{\si} = \bW_{\si} \hbphi^{(\si)}, \text{and } \by = \bW \hbpsi, \by_{\si} = \bW_{\si} \hbpsi^{(\si)}, \label{eq:submodels}
     \end{align}
      where all the coefficient estimators are minimum $\ell_2$-norm OLS estimators, i.e., fully regularized OLS estimators. By Lemma \ref{lemma:shen_cor2}, we have $\hbphi^{(\si)} = (\bI_q - \bP_i)\hbphi$ and $\hbpsi^{(\si)} = (\bI_q - \bP_i)\hbpsi$.
    
    Let $\Gram_{\bW_{\si}} = [\bW_{\si} (\bW_{\si})^\top]^\dagger$. We then use the expression in \eqref{eq:fwl.jc.hd.gls} to express $\hbtau^{(\si)}$:

    \begin{align}
        \hbtau^{(\si)} &= \big[(\bT_{\si})^\top \Gram_{\bW_{\si}} \bT_{\si}\big]^\dagger (\bT_{\si})^\top \Gram_{\bW_{\si}} \by_{\si}
    \end{align}

    Note that 

    \begin{align*}
        (\bT_{\si})^\top \Gram_{\bW_{\si}} \by_{\si} 
        &= (\bT_{\si})^\top (\bW_{\si})^{\dagger,\top} (\bW_{\si})^\dagger\by_{\si} \\
        &= \hbphi^{(\si)\top} \hbpsi^{(\si)} \\
        &= \hbphi^{\top} (\bI_q - \bP_i) \hbpsi \\
        &= \bT^\top \bW^{\top, \dagger} (\bI_q - \bP_i) \bW^\dagger \by \\
        &= \bT^\top \Gram_{\bW} (\bI_n - \bQ_i) \by 
        \quad \because \eqref{eq:prep.2}\\
        &= \bT^\top (\bI_n - \bQ_i)^\top \Gram_{\bW} (\bI_n - \bQ_i) \by 
        \quad \because \eqref{eq:prep.3}\\
        &= \bT^\top (\bI_n - \bQ_i)^\top \bW^{\dagger,\top} \bW^\dagger (\bI_n - \bQ_i) \by \\
        &= (\bW^\dagger (\bI_n - \bQ_i)\bT)^\top \bW^\dagger (\bI_n - \bQ_i) \by
    \end{align*}

    Similarly, $(\bT_{\si})^\top \Gram_{\bW_{\si}} \bT_{\si} 
    = (\bW^\dagger (\bI_n - \bQ_i)\bT)^\top \bW^\dagger (\bI_n - \bQ_i) \bT$.
    Therefore, we have

    \begin{align*}
        \hbtau^{(\si)} 
        & = \big[(\bT_{\si})^\top \Gram_{\bW_{\si}} \bT_{\si}\big]^\dagger (\bT_{\si})^\top \Gram_{\bW_{\si}} \by_{\si} \\
        & = [(\bW^\dagger (\bI_n - \bQ_i)\bT)^\top \bW^\dagger (\bI_n - \bQ_i) \bT]^\dagger (\bW^\dagger (\bI_n - \bQ_i)\bT)^\top \bW^\dagger (\bI_n - \bQ_i) \by \\
        & = [\bW^\dagger (\bI_n - \bQ_i)\bT]^\dagger \bW^\dagger (\bI_n - \bQ_i) \by 
        \quad \because \eqref{eq:pseudo_compute}\\
        & = [(\bI_q - \bP_i)\bW^\dagger\bT]^\dagger (\bI_q - \bP_i)\bW^\dagger \by \quad \because \eqref{eq:prep.2}
    \end{align*}

    Next, we derive $\hblambda^{(\si)}$ using the expression in \eqref{eq:fwl.j.hd.alt}:

    \begin{align}
        \hblambda^{(\si)} &= (\bW_{\si})^\top  \Gram_{\bW_{\si}}(\by_{\si} - \bT_{\si}\hbtau^{(\si)})  \nonumber\\
        &= \big[(\bW_{\si})^\dagger \bW_{\si}\big]^\top (\bW_{\si})^\dagger \by_{\si} - \big[(\bW_{\si})^\dagger \bW_{\si}\big]^\top (\bW_{\si})^\dagger\bT_{\si}\hbtau^{(\si)} \nonumber\\
        &=\bW^\dagger \bW (\bI_q - \bP_i)\hbpsi - \bW^\dagger \bW (\bI_q - \bP_i)\hbphi \hbtau^{(\si)} 
        \quad \because \eqref{eq:prep.4} \nonumber\\
        &=\bW^\top \Gram_{\bW} \bW (\bI_q - \bP_i) \bW^\dagger \by - \bW^\top \Gram_{\bW} \bW (\bI_q - \bP_i) \bW^\dagger \bT \hbtau^{(\si)} \quad (\because \bW^\dagger\bW=\bW^\top \Gram_{\bW} \bW) \nonumber\\
        & = \bW^\top \Gram_{\bW} (\bI_n - \bQ_i)\by - \bW^\top \Gram_{\bW} (\bI_n - \bQ_i)\bT\hbtau^{(\si)} 
        \quad \because \eqref{eq:prep.2} \label{eq:lambda_loo_mid}\\
        & = \bW^\top (\bI_n - \bQ_i)^\top \Gram_{\bW} (\bI_n - \bQ_i)\by - \bW^\top (\bI_n - \bQ_i)^\top \Gram_{\bW} (\bI_n - \bQ_i)\bT \hbtau^{(\si)}  
        \quad \because \eqref{eq:prep.3}\nonumber\\
        & = \bW^\top [\bW^\dagger(\bI_n - \bQ_i)]^\top \bW^\dagger (\bI_n - \bQ_i)\by - \bW^\top [\bW^\dagger(\bI_n - \bQ_i)]^\top \bW^\dagger (\bI_n - \bQ_i)\bT [\bW^\dagger (\bI_n - \bQ_i)\bT]^\dagger \bW^\dagger (\bI_n - \bQ_i) \by  \nonumber\\
         & = \bW^\top [\bW^\dagger(\bI_n - \bQ_i)]^\top \bW^\dagger (\bI_n - \bQ_i)\by - \bW^\top [\bW^\dagger(\bI_n - \bQ_i)]^\top \bP_{\bW^\dagger (\bI_n - \bQ_i)\bT}  \bW^\dagger (\bI_n - \bQ_i) \by  \nonumber\\
         & = \bW^\top [\bW^\dagger(\bI_n - \bQ_i)]^\top \bP_{\bW^\dagger (\bI_n - \bQ_i)\bT}^\perp \bW^\dagger (\bI_n - \bQ_i) \by \nonumber\\
         & = \bW^\top [(\bI_q - \bP_i)\bW^\dagger]^\top \bP_{(\bI_q - \bP_i)\bW^\dagger\bT}^\perp (\bI_q - \bP_i)\bW^\dagger \by \quad \because \eqref{eq:prep.2}     \nonumber
    \end{align}
\end{proof}

\subsection{Proof of Corollary~\ref{cor:loo.residual}}

Before proving this corollary, we first establish a useful expression: 

\begin{equation}
    \bw_i^\top \bW^\dagger \bQ_i = \be_i^\top \bD^{-1} \Gram_{\bW}, \label{eq:pf_diag}
\end{equation}
where $\bD = diag (\Gram_{\bW})$ and $\bQ_i$ is defined as in \eqref{mat:QiPi}. To see this, note that:

\begin{align*}
    \bw_i^\top \bW^\dagger \bQ_i &= \frac{\bw_i^\top \bW^\dagger \be_i \be_i^\top \Gram_{\bW}}{\be_i^\top \Gram_{\bW} \be_i} \\
    &= \frac{\bw_i^\top \bW^\dagger \bW^{\top, \dagger} \bw_i \be_i^\top \Gram_{\bW}}{\be_i^\top \Gram_{\bW} \be_i} \quad (\because \bw_i = \bW^\top \be_i)  \\
    &= \frac{\be_i^\top \Gram_{\bW}}{\be_i^\top \Gram_{\bW}\be_i} \quad (\because \bw_i^\top \bW^\dagger \bW^{\top, \dagger} \bw_i = 1)\\
    &= \be_i^\top \bD^{-1} \Gram_{\bW} \quad ( \because \bD^{-1}\be_i = \frac{\be_i}{\be_i^\top \bD \be_i})
\end{align*}

Then we are ready to prove Corollary \ref{cor:loo.residual}:

\begin{proof}[Proof of Corollary \ref{cor:loo.residual}]

We utilize \eqref{eq:fwl.j.hd.alt} and \eqref{eq:lambda_loo_mid}, which provide alternative expressions for $\hblambda$ and $\hblambda^{(\si)}$, to prove:

\begin{align*}
    \tbvarepsilon_i &= \by_i-\hby_i \\
    &= \bw_i^\top(\hblambda-\hblambda^{(\si)}) + \bt_i^\top(\hbtau-\hbtau^{(\si)}) \\
    &= \bw_i^\top \left(\bW^\top \Gram_{\bW} (\by - \bT \hbtau) - \bW^\top \Gram_{\bW} (\bI_n - \bQ_i) (\by - \bT\hbtau^{(\si)})\right) + \bt_i^\top (\hbtau - \hbtau^{(\si)}) \\
    & = \bw_i^\top \left(\bW^\dagger \by - \bW^\dagger \bT \hbtau - \bW^\dagger \by +\bW^\dagger\bT \hbtau^{(\si)} + \bW^\dagger \bQ_i\by - \bW^\dagger\bQ_i\bT\hbtau^{(\si)}\right) + \bt_i^\top (\hbtau - \hbtau^{(\si)}) \\
    &= \bw_i^\top \bW^\dagger \bQ_i (\by - \bT\hbtau^{(\si)}) - \bw_i^\top \bW^\dagger\bT (\hbtau-\hbtau^{(\si)}) + \bt_i^\top (\hbtau - \hbtau^{(\si)}) \\
    &= \be_i^\top \bD^{-1} \Gram_{\bW} (\by - \bT\hbtau^{(\si)}) \quad (\because \eqref{eq:pf_diag};\bw_i^\top \bW^\dagger \bT = \bw_i^\top \hbphi = \bt_i^\top, \text{ where }\hbphi \text{ is defined as in } \eqref{eq:submodels}) \\
    &= \be_i^\top \bD^{-1} \Gram_{\bW} (\bI_n - \bH_i) \by \quad \left(\text{substitute } \hbtau^{(\si)}; \bH_i = \bT [(\bI_q - \bP_i)\bW^\dagger \bT]^\dagger (\bI_q - \bP_i)\bW^\dagger\right)                                                                                                                                                                                                                                                                                                                                                                                                                                      
\end{align*}
    
\end{proof}

\section{Deferred proof from section \ref{sec:stat_infer}}\label{apdx:stat_infer}

We provide the following equation to facilitate the proof:

If Assumption \ref{assump:lm} holds with $\bSigma = \sigma^2 \bI$, then we have for any deterministic matrix $\bM \in \RR^{n \times n}$
\begin{align}
	\bbE[ \by^\top \bM \by] 
        &= \bbE\left[ \left(\bX \bbeta + \bvarepsilon \big)^\top \bM \big(\bX \bbeta + \bvarepsilon \right) \right] 
        \nonumber\\
        &= \bbeta^\top \bX^\top \bM \bX \bbeta + \bbE \left[ \bvarepsilon^\top \bM \bvarepsilon  \right]
        \nonumber\\
        &= \bbeta^\top \bX^\top \bM \bX \bbeta + \sigma^2 \cdot \tr(\bM),   \label{eqn:exp_quadratic}
\end{align}
where the last equality follows from 
\[
    \bbE [ \bvarepsilon^\top \bM \bvarepsilon  ] = \bbE [ \tr( \bvarepsilon^\top \bM \bvarepsilon) ] = \bbE [ \tr( \bM \bvarepsilon \bvarepsilon^\top ) ] = \tr( \bM \cdot \bbE [  \bvarepsilon \bvarepsilon^\top ] ) = \sigma^2 \cdot \tr(\bM).
\]

\subsection{Proof of Theorem \ref{thm:var.estimator.pr}}

Recall from Theorem \ref{cor:loo.residual} that $\tbvarepsilon_i =\be_i^\top \big[ \diag(\Gram_{\bW}) \big]^{-1} \cdot \Gram_{\bW} (\bI_n - \bH_i)  \by$, where $\bH_i = \bT [(\bI_q - \bP_i)\bW^\dagger \bT]^\dagger (\bI_q - \bP_i)\bW^\dagger$, $\bP_i = \frac{\bW^\dagger \be_i \be_i^\top \bW^{\dagger, \top}}{\be_i^\top \Gram_{\bW} \be_i}$. We apply \eqref{eqn:exp_quadratic} to obtain

\begin{align*}
	\bbE[ \tbvarepsilon_i^2] 
        &= \bbE[ \by^\top (\bI_n-\bH_i)^\top \Gram_{\bW}\big[ \diag(\Gram_{\bW}) \big]^{-1}\be_i
        \be_i^\top \big[ \diag(\Gram_{\bW}) \big]^{-1} \Gram_{\bW}(\bI_n-\bH_i)\by ]\\
        &= \bbeta^\top \bX^\top \bM_i \bX \bbeta 
        +
        \sigma^2 \cdot \tr \left( \bM_i\right) \\
        &\left(\text{where } \bM_i = (\bI_n-\bH_i)^\top \Gram_{\bW}\big[ \diag(\Gram_{\bW}) \big]^{-1}\be_i
        \be_i^\top\big[ \diag(\Gram_{\bW}) \big]^{-1} \Gram_{\bW}(\bI_n-\bH_i)\right)\\
        &= \bbE[\tbvarepsilon_i]^2 + \sigma^2  \left\|\be_i^\top \big[ \diag(\Gram_{\bW}) \big]^{-1} \Gram_{\bW}(\bI_n-\bH_i) \right\|_2^2
\end{align*}
where the last equality follows from the observation that $\bbE[\tbvarepsilon_i] = \be_i^\top\big[ \diag(\Gram_{\bW}) \big]^{-1} \Gram_{\bW}(\bI_n-\bH_i) \bX \bbeta$.

Take summation over $i$ on both sides, we have 

\begin{align}
    \bbE\big[ \|\tbvarepsilon\|_2^2] &=  \| \bbE[\tbvarepsilon]\|_2^2 + \sigma^2 \Sigma_{i=1}^n \left\|\be_i^\top \big[ \diag(\Gram_{\bW}) \big]^{-1} \Gram_{\bW}(\bI_n-\bH_i) \right\|_2^2
\end{align}
where the $i$-th element of $\tbvarepsilon$ is $\tbvarepsilon_i$.
Therefore, 

\begin{align}
     \bbE\big[ \frac{\|\tbvarepsilon\|_2^2}{\Sigma_{i=1}^n \left\|\be_i^\top\big[ \diag(\Gram_{\bW}) \big]^{-1} \Gram_{\bW}(\bI_n-\bH_i) \right\|_2^2}\big] &=  \frac{\| \bbE[\tbvarepsilon]\|_2^2}{\Sigma_{i=1}^n \left\|\be_i^\top\big[ \diag(\Gram_{\bW}) \big]^{-1} \Gram_{\bW}(\bI_n-\bH_i) \right\|_2^2} + \sigma^2 
\end{align}

\subsection{Proof of Theorem~\ref{thm:var.estimator.partial}} \label{sec:proof.var.partial_reg}

\begin{proof}[Proof of Theorem~\ref{thm:var.estimator.partial}] 
We first prove \eqref{eq:bias.j.hd}. 
In this pursuit, we plug $\hbbeta_{\cW}^{[\cW]}$ in \eqref{eq:beta.fwl.hd} into the numerator of \eqref{eq:var.estimator.partial.hd.j} to obtain 
\begin{align}
	\| \by - \bP_{\bT}^\perp \bW \cdot \hbbeta_{\cJ}^{[\cJ]} \|_2^2
	&= \| \{ \bI_n - (\bP_{\bT}^\perp \bW)(\bP_{\bT}^\perp \bW)^\dagger \} \cdot \by \|_2^2 \nonumber
	\\
	&= \by^\top \cdot \{ \bI_n - (\bP_{\bT}^\perp \bW)(\bP_{\bT}^\perp \bW)^\dagger \} \cdot \by\\
        &= \by^\top \cdot \{ \bI_n - \bP_{\bT}^\perp \bW \bW^\dagger \bP_{\bT}^\perp\} \cdot \by\\
        &= \by^\top \cdot\bP_{\bT}\cdot \by\label{eq:bias.partial.1} 
\end{align} 
Applying \eqref{eqn:exp_quadratic} with $\bM = \bP_{\bT}$ to \eqref{eq:bias.partial.1} and observing $\bbE[\by] = \bX \bbeta$ completes this part of the proof. 

Next, we prove \eqref{eq:bias.jc.hd} following the proof strategy above. 
That is, we plug $\hbbeta_{\cW^c}^{[\cW]}$ in \eqref{eq:beta.fwl.hd} into the numerator of \eqref{eq:var.estimator.partial.hd.jc} to obtain 
\begin{align}
    \| \bW^\dagger \by - \bW^\dagger \bT \cdot \hbbeta_{\cJ^c}^{[\cJ]} \|_2^2
    &= \| \{ \bI_{q} - (\bW^\dagger \bT)(\bW^\dagger \bT)^\dagger \} \bW^\dagger \cdot \by \|_2^2
    \nonumber\\
    &= \by^\top \cdot (\bW^\dagger)^\top \{ \bI_{q} - (\bW^\dagger \bT)(\bW^\dagger \bT)^\dagger \} \bW^\dagger \cdot \by. \label{eq:bias.partial.2} 
\end{align} 
Applying \eqref{eqn:exp_quadratic} with $\bM = (\bW^\dagger)^\top \bP_{\bW^\dagger\bT}^\perp \bW^\dagger$ to \eqref{eq:bias.partial.2} and observing $\bbE[\by] = \bX \bbeta$ completes the proof. 

Therefore, $\bbE[\hsigma_{\cW^c}^2] = \sigma^2+\frac{ \| \bP_{\bW^\dagger\bT}^\perp \cdot \bW^\dagger \cdot \bbE[\by] \|_2^2}{\tr(\bP_{\bW^\dagger\bT}^\perp \cdot \{ \bW^\top \bW \}^\dagger)}$
\end{proof}

\section{Simulation details to Section \ref{sec:simulation}}

Recall that we omitted the results for $\hsigmareg^2$ in the main text. Here, we present the results for all four variance estimators together and provide a more detailed analysis.

\subsection{Simulation I: fixed covariate size $p$ and varying sample size $n$}

As shown in the first row of Figure \ref{fig:bias_sim1_supp}, the bias of $\hsigma_\cW^2$ is large and increases significantly with the sample size, along with a gradual rise in standard deviation.
Zooming in, we observe that for the standard normal model, the bias decreases for all three estimators as $n$ increases, with $\hsigmap^2$ exhibiting the smallest bias.
In both the spiked and geometric models, the other three estimators show near-zero average bias across all sample sizes. Additionally, their variances decrease as $n$ increases, with $\hsigma_{\cW^c}^2$ showing a smaller variance in the geometric model.

\begin{figure*}[h!]
	\centering  
	\begin{subfigure}{0.30\linewidth}
		\centering 
		\includegraphics[width=\textwidth]
		{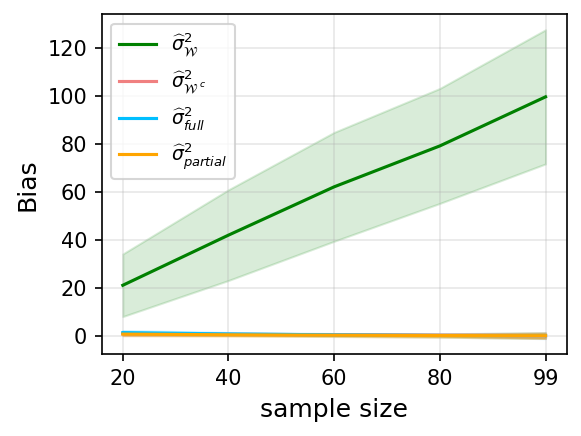}
	\end{subfigure} 
        \begin{subfigure}{0.30\linewidth}
		\centering 
		\includegraphics[width=\textwidth]
		{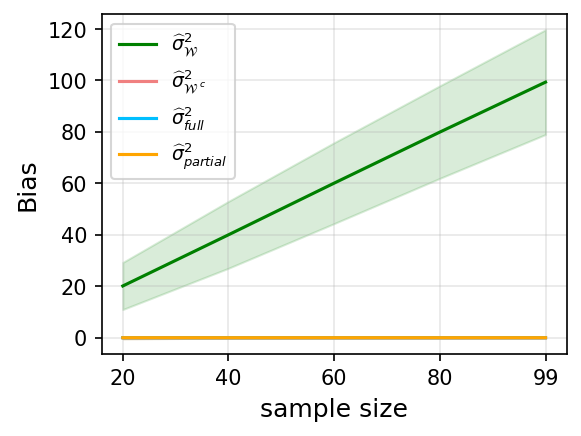}
	\end{subfigure}
        \begin{subfigure}{0.30\linewidth}
		\centering 
		\includegraphics[width=\textwidth]
		{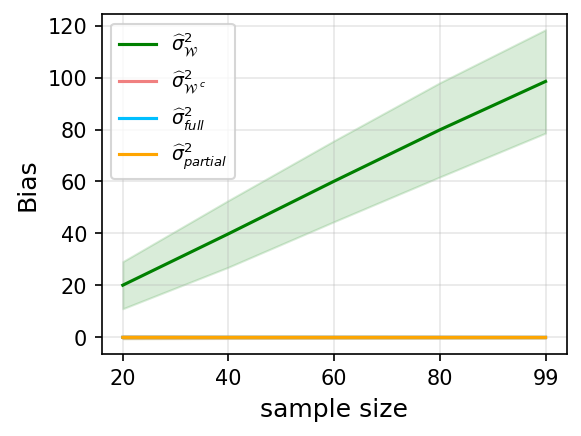}
	\end{subfigure}

        \begin{subfigure}{0.30\linewidth}
		\centering 
		\includegraphics[width=\textwidth]
		{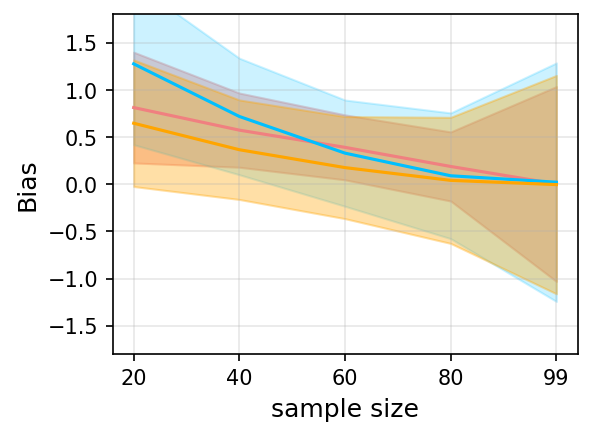}
		\caption{Standard normal model.} 
	\end{subfigure} 
        \begin{subfigure}{0.30\linewidth}
		\centering 
		\includegraphics[width=\textwidth]
		{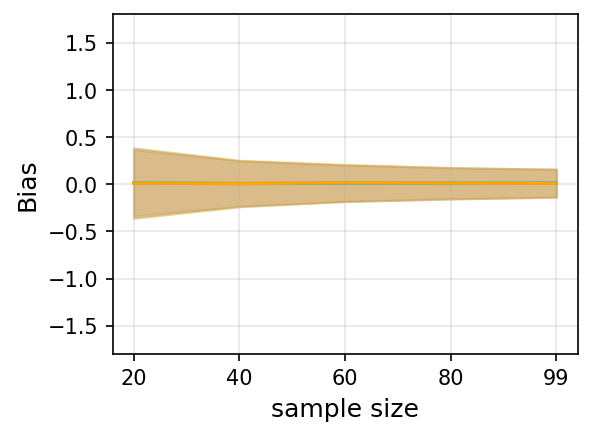}
		\caption{Spiked model.} 
	\end{subfigure}
        \begin{subfigure}{0.30\linewidth}
		\centering 
		\includegraphics[width=\textwidth]
		{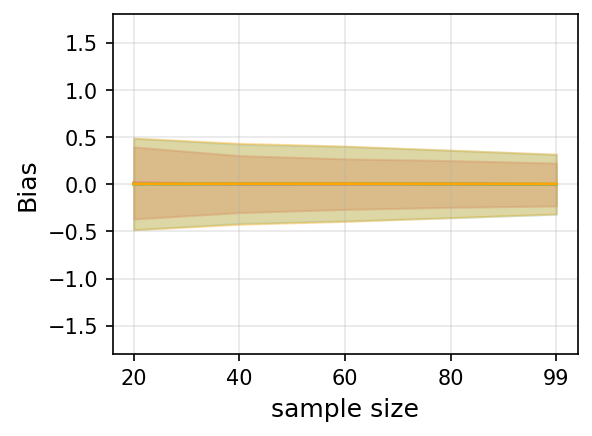}
		\caption{Geometric model.} 
	\end{subfigure}
 \vspace{.2in}
	\caption{Simulation results with fixed $p = 100$ and varying $n \in \{20, 40, 60, 80, 99\}$. 
    The second row shows a zoomed-in view of the first row. Solid lines represent the mean over 100 trials, with shading indicating $\pm$ one standard error.
    }
	\label{fig:bias_sim1_supp} 
\end{figure*}

\subsection{Simulation II: fixed $(n,p)$ ratio and varying covariate size $p$}

Figure~\ref{fig:bias_sim2_supp} visualizes the biases $\hsigma^2 - \sigma^2$ for each variance estimator and covariate model across different covariate sizes.
Similar to Simulation I, the bias of $\hsigmareg^2$ increases with covariate size across all three models.
For the spiked and geometric models, the average biases of all three estimators are nearly zero across different covariate sizes.
In contrast, for the standard normal model, the biases are noticeably positive for all three estimators, with $\hsigmap^2$ showing the smallest bias and $\hsigma_{\cW^c}^2$ the largest.
In the geometric model, the variances of $\hsigmap^2$ and $\hsigmaf^2$ increase with $p$, while $\hsigma_{\cW^c}^2$ remains nearly constant. In the standard normal and spiked models, the variances of all three estimators decrease as $p$ increases.
Moreover, $\hsigma_{\cW^c}^2$ exhibits the smallest variance among the three estimators in both the standard normal and geometric models, indicating its stability and precision.

\begin{figure*}[h!]
	\centering 
	\begin{subfigure}{0.32\linewidth}
		\centering 
		\includegraphics[width=\textwidth]
		{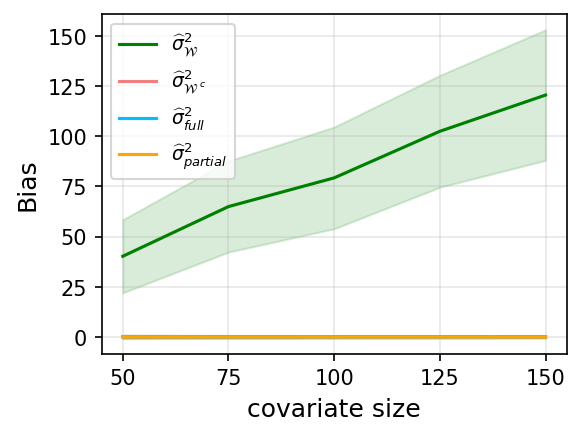}
	\end{subfigure} 
        \begin{subfigure}{0.32\linewidth}
		\centering 
		\includegraphics[width=\textwidth]
		{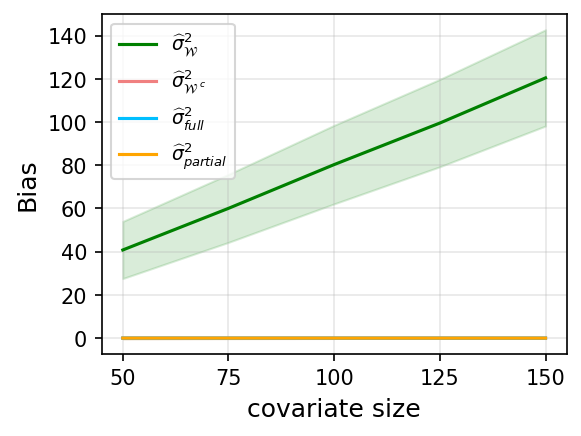} 
	\end{subfigure}
        \begin{subfigure}{0.32\linewidth}
		\centering 
		\includegraphics[width=\textwidth]
		{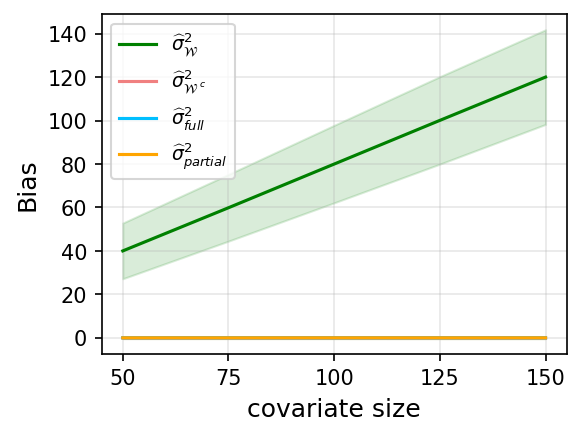}
	\end{subfigure}
        \begin{subfigure}{0.32\linewidth}
		\centering 
		\includegraphics[width=\textwidth]
		{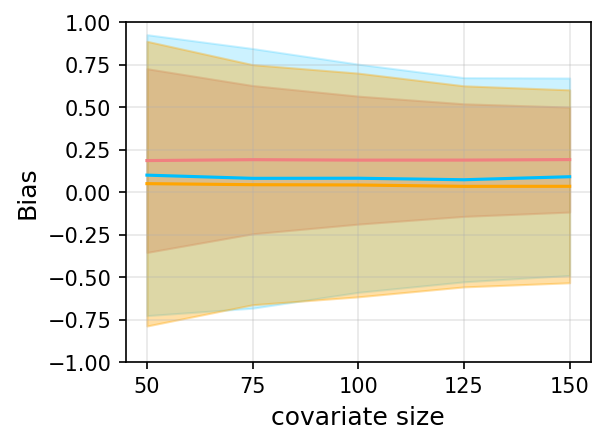}
		\caption{Standard normal model.} 
	\end{subfigure} 
        \begin{subfigure}{0.32\linewidth}
		\centering 
		\includegraphics[width=\textwidth]
		{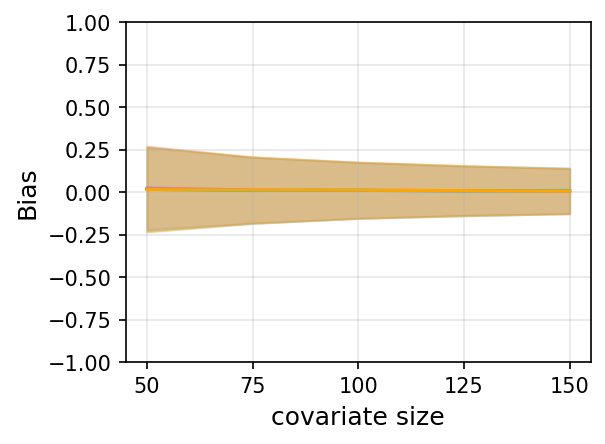}
		\caption{Spiked model.} 
	\end{subfigure}
        \begin{subfigure}{0.32\linewidth}
		\centering 
		\includegraphics[width=\textwidth]
		{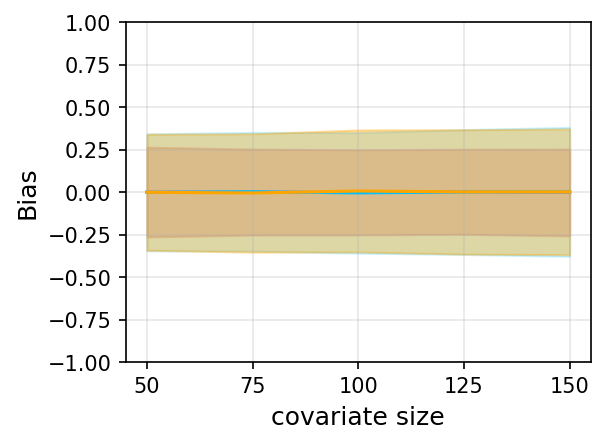}
		\caption{Geometric model.} 
	\end{subfigure}
	\caption{Simulation results displaying the biases of four variance estimators with fixed aspect ratio $n/p = 0.8$ and varying $p \in \{50,75,100,125,150 \}$. 
	The second row shows a zoomed-in view of the first row. The solid lines represent the mean over 100 trials, with shading indicating $\pm$ one standard error.} 
	\label{fig:bias_sim2_supp} 
\end{figure*}

\subsection{Simulation III: fixed $(n,p)$ and increasing noise variance $\sigma^2$}

Figure~\ref{fig:bias_sim3_supp} shows the biases $\hsigma^2 - \sigma^2$ for each variance estimator and covariate model across increasing noise levels.
$\hsigma_\cW^2$ maintains a consistently high bias with increasing variance. In the zoomed-in plots, we observe that the average biases of $\hsigma_{\cW^c}^2$, $\hsigmaf^2$, and $\hsigmap^2$ are nearly zero across all models, though their variances grow with higher noise levels. Despite this, $\hsigma_{\cW^c}^2$ consistently has the smallest variance among the three estimators in both the standard normal and geometric models.

\subsection{Simulation IV: fixed $(n,p)$ and increasing intercept $\beta_0$}

Figure~\ref{fig:bias_sim4_supp} shows the biases $\hsigma^2 - \sigma^2$ for each variance estimator and covariate model as the intercept magnitude increases.
As before, $\hsigma_\cW^2$ performs poorly, with bias increasing as the intercept magnitude grows.
In the zoomed-in plots, the bias of $\hsigmaf^2$ increases significantly with intercept magnitude in the standard normal model.
Although $\hsigma_{\cW^c}^2$ shows slightly more bias than $\hsigmap^2$, whose bias is nearly zero, it has smaller variance.
In both the spiked and geometric models, the average biases of $\hsigma_{\cW^c}^2$, $\hsigmaf^2$, and $\hsigmap^2$ are nearly zero, with their variances remaining stable.

\subsection{Summary}

In addition to the general takeaway in Remark \ref{rmk:simu_takeaway}, a comparison of Simulation III with the other three simulations reveals that the bias $\bbE[ \hsigma^2 ] - \sigma^2$ primarily depends on the design matrix $\bX = [\bW,\bone]$ and the parameter $\bbeta = (\bbeta_1^\top, \beta_0)^\top$, as shown by the bias terms in \eqref{eq:bias.hd}, \eqref{eq:bias.hd.partial}, \eqref{eq:bias.j.hd}, and \eqref{eq:bias.jc.hd}.

The variance estimator $\hsigma_\cW^2$, based on in-sample residuals from partial regression induced by $\hbbeta_\cW^{[\cW]}$, consistently shows high bias across all simulation settings and generative models.
In contrast, $\hsigma_{\cW^c}^2$, the estimator based on in-sample residuals induced by $\hbbeta_{\cW^c}^{[\cW]}$, performs well, exhibiting nearly zero bias in the spiked and geometric models and the lowest variance across most models and simulations.

Under the high-dimensional regime, full regularization induces implicit regularization across all coefficients, whereas partial regularization leaves the constant column unregularized.
Our observations show that $\hsigmaf^2$ has higher bias in the standard normal model (Simulations I, II, and IV), indicating its limitations in handling isotropic, unstructured data compared to $\hsigmap^2$. However, both estimators perform similarly in the spiked and geometric models, suggesting comparable accuracy in estimating noise variance in structured contexts.

\begin{figure*}[h!]
	\centering 
		\begin{subfigure}{0.32\linewidth}
		\centering 
		\includegraphics[width=\textwidth]
		{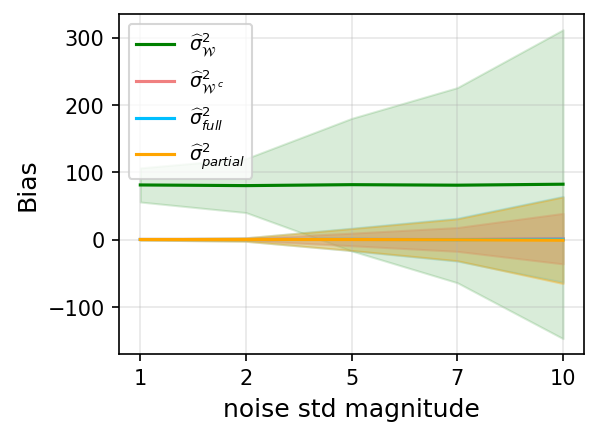}
	\end{subfigure} 
        \begin{subfigure}{0.32\linewidth}
		\centering 
		\includegraphics[width=\textwidth]
		{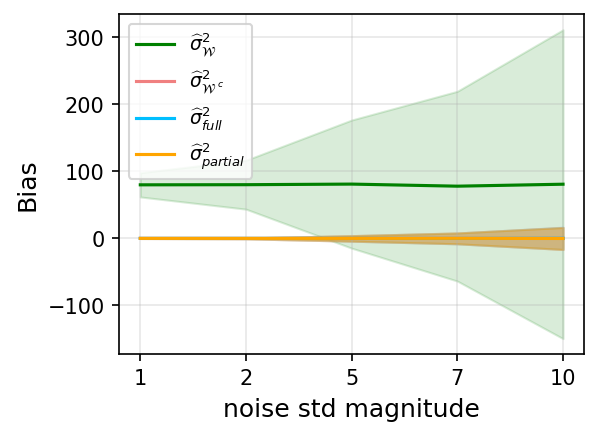}
	\end{subfigure}
        \begin{subfigure}{0.32\linewidth}
		\centering 
		\includegraphics[width=\textwidth]
		{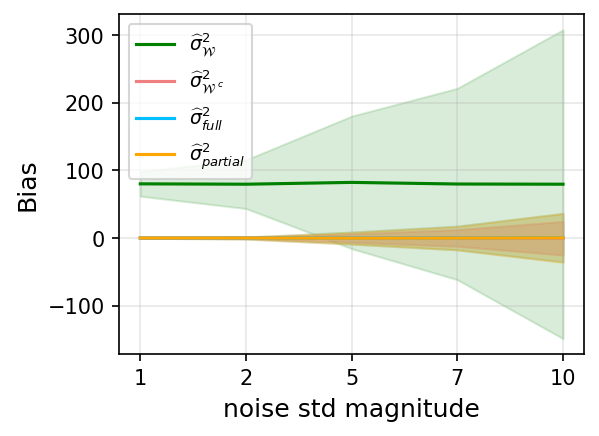}
	\end{subfigure}
        \begin{subfigure}{0.32\linewidth}
		\centering 
		\includegraphics[width=\textwidth]
		{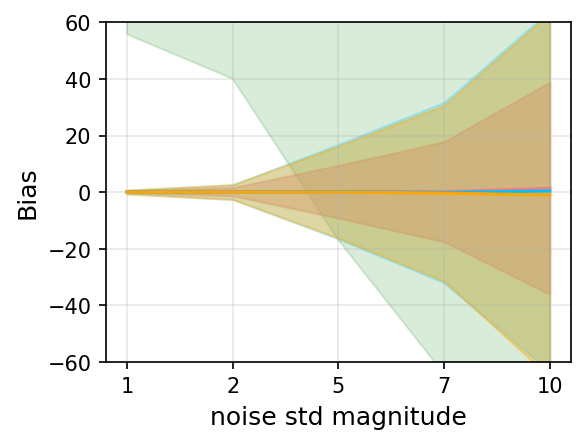}
		\caption{Standard normal model.} 
	\end{subfigure} 
        \begin{subfigure}{0.32\linewidth}
		\centering 
		\includegraphics[width=\textwidth]
		{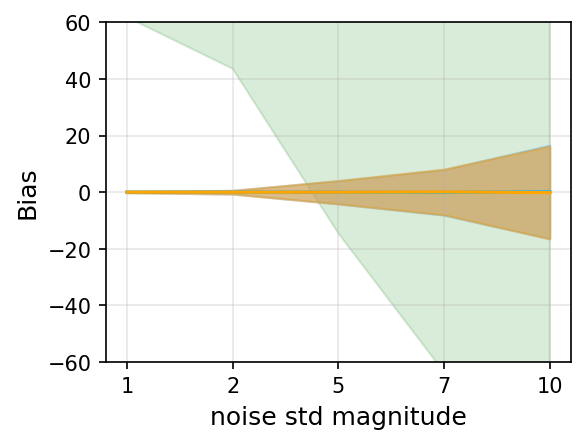}
		\caption{Spiked model.} 
	\end{subfigure}
        \begin{subfigure}{0.32\linewidth}
		\centering 
		\includegraphics[width=\textwidth]
		{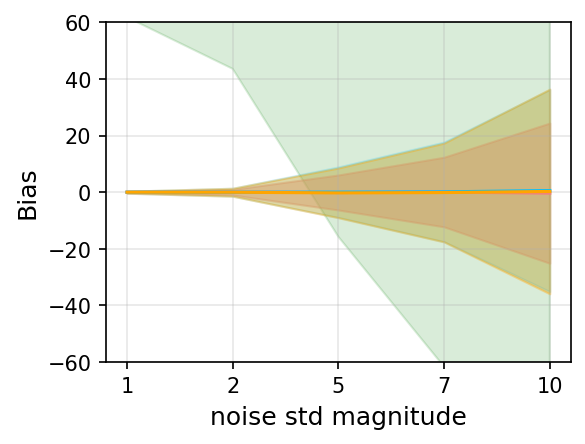}
		\caption{Geometric model.} 
	\end{subfigure}
	\caption{Simulation results displaying the biases of our variance estimator in \eqref{eq:var.estimator.hd} with fixed $p=100$, $n=80$, and varying $\sigma \in \{1, 2, 5, 7, 10 \}$. 
	The second row shows a zoomed-in view of the first row. The solid lines represent the mean over 100 trials, with shading indicating $\pm$ one standard error.}
	\label{fig:bias_sim3_supp} 
\end{figure*}

\begin{figure*}[h!]
	\centering 
		\begin{subfigure}{0.30\linewidth}
		\centering 
		\includegraphics[width=\textwidth]
		{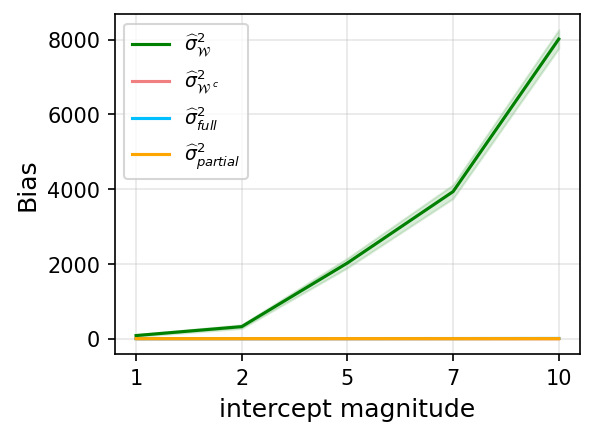}
	\end{subfigure} 
        \begin{subfigure}{0.30\linewidth}
		\centering 
		\includegraphics[width=\textwidth]
		{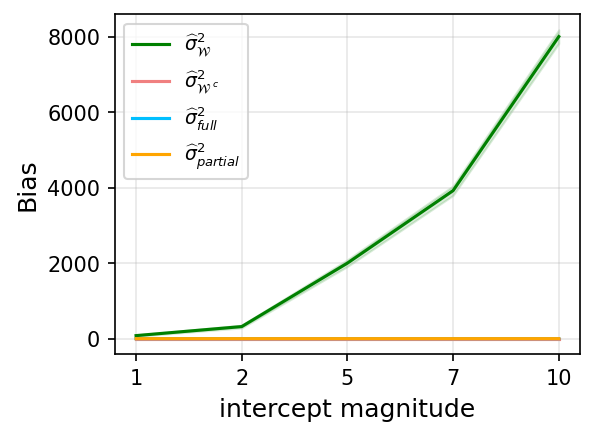}
	\end{subfigure}
        \begin{subfigure}{0.30\linewidth}
		\centering 
		\includegraphics[width=\textwidth]
		{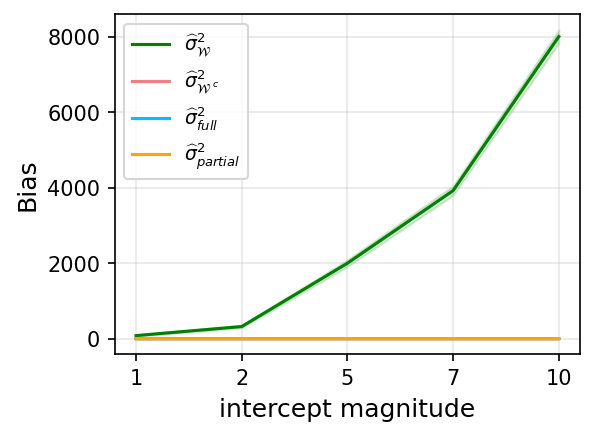}
	\end{subfigure}
        \begin{subfigure}{0.30\linewidth}
		\centering 
		\includegraphics[width=\textwidth]
		{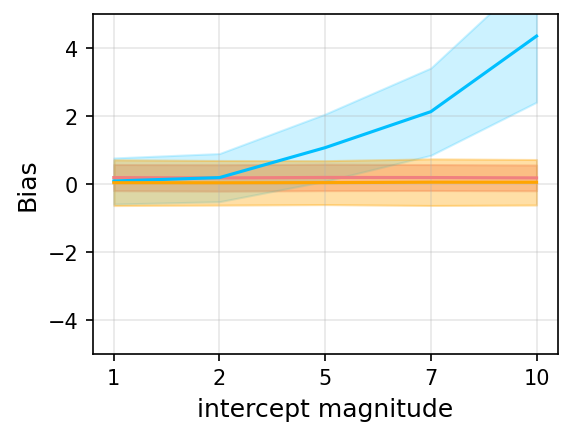}
		\caption{Standard normal model.} 
	\end{subfigure} 
        \begin{subfigure}{0.30\linewidth}
		\centering 
		\includegraphics[width=\textwidth]
		{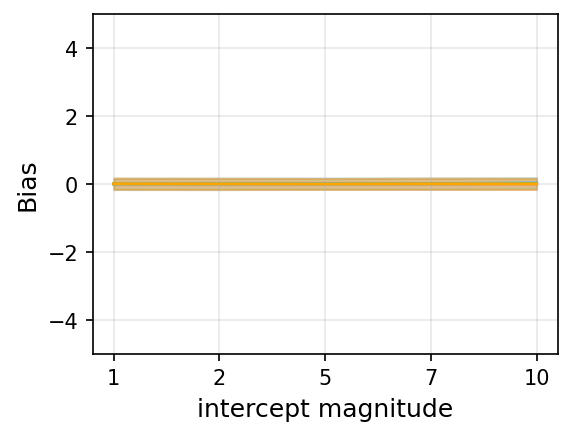}
		\caption{Spiked model.} 
	\end{subfigure}
        \begin{subfigure}{0.30\linewidth}
		\centering 
		\includegraphics[width=\textwidth]
		{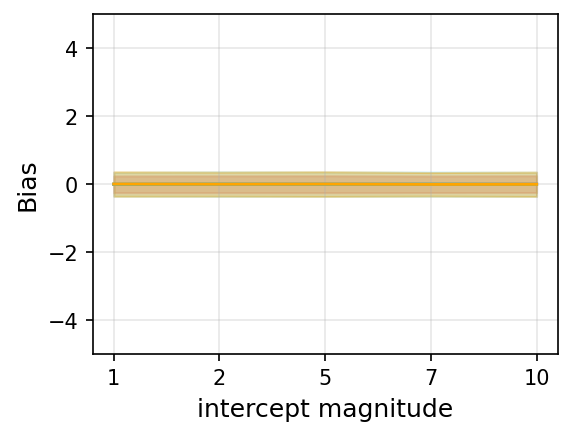}
		\caption{Geometric model.} 
	\end{subfigure}
	\caption{Simulation results displaying the biases of our variance estimator in \eqref{eq:var.estimator.hd} with fixed $p=100$, $n=80$, and varying intercept $\beta_0 \in \{1, 2, 5, 7, 10 \}$. 
	The second row shows a zoomed-in view of the first row. The solid lines represent the mean over 100 trials, with shading indicating $\pm$ one standard error.}
	\label{fig:bias_sim4_supp} 
\end{figure*}

\end{document}